\documentclass[11pt,reqno]{amsart}

\usepackage{amssymb}
\usepackage{latexsym}
\usepackage{amsbsy}
\usepackage{amsfonts}
\usepackage{appendix}
\usepackage{pstricks}
\usepackage{pst-fill,pst-grad}
\usepackage{color}
\usepackage{enumitem}

\numberwithin{equation}{section}
\newtheorem{theorem}{Theorem}[section]
\newtheorem{corollary}[theorem]{Corollary}
\newtheorem{lemma}[theorem]{Lemma}
\newtheorem{proposition}[theorem]{Proposition}

\theoremstyle{definition}
\newtheorem{definition}[theorem]{Definition}
\theoremstyle{remark}

\newtheorem{remark}[theorem]{Remark}

\newtheorem{example}[theorem]{Example}

\newdimen\AAdi%
\newbox\AAbo%
\def\AAk#1#2{\setbox\AAbo=\hbox{#2}\AAdi=\wd\AAbo\kern#1\AAdi{}}%

\makeatletter
\def\eqlabel#1{\def\@currentlabel{#1}}

\def\formula#1{\def\@tempa{#1}\let\@tempb\theequation\def\theequation{%
\hbox{#1}}\def\@currentlabel{(\theequation)}$$}
\def\endformula{\leqno\hbox{(\@tempa)}$$\@ignoretrue\let\theequation\@tempb}

\def\given{\hskip5\p@\relax\vrule\@width.4\p@\hskip5\p@\relax}

\newcommand{\open}[1]{%
\par\normalfont\topsep6\p@\@plus6\p@\trivlist\item[\hskip\labelsep\itshape#1%
\@addpunct{.}]\ignorespaces}

\DeclareRobustCommand{\close}[1]{%
  \ifmmode 
  \else \leavevmode\unskip\penalty9999 \hbox{}\nobreak\hfill
  \fi
  \quad\hbox{$#1$}}
\makeatother

\newlength{\toskip}

\def\<{\langle}
\def\>{\rangle}

\def \Var {\textrm{Var}}

\def \Osc {\textrm{Osc}}

\begin{document}
\date{\today}

\title[Integrated $\Gamma 2$]{A journey with the integrated $\Gamma 2$ criterion and its weak forms.} 
\bigskip

 \author[P. Cattiaux]{\textbf{\quad {Patrick} Cattiaux $^{\spadesuit}$ \, \, }}
\address{{\bf {Patrick} CATTIAUX},\\ Institut de Math\'ematiques de Toulouse. CNRS UMR 5219. \\
Universit\'e Paul Sabatier,
\\ 118 route
de Narbonne, F-31062 Toulouse cedex 09.} \email{patrick.cattiaux@math.univ-toulouse.fr}

\author[A. Guillin]{\textbf{\quad {Arnaud} Guillin $^{\diamondsuit}$}}
\address{{\bf {Arnaud} GUILLIN},\\Universit\'e Clermont Auvergne, CNRS, LMBP, F-63000 CLERMONT-FERRAND, FRANCE.} \email{arnaud.guillin@uca.fr}

\maketitle

 \begin{center}

 \textsc{$^{\spadesuit}$  Universit\'e de Toulouse}
\smallskip

\smallskip

\textsc{$^{\diamondsuit}$ Universit\'e Clermont-Auvergne}
\smallskip

\end{center}

\begin{abstract}
As the title indicates this paper will describe several extensions and applications of the $\Gamma_2$ integrated criterion introduced by M. Ledoux following ideas of B. Hellffer. We introduce general weak versions and show that they are equivalent to the weak Poincar\'e inequalities introduced by M. R\"{o}ckner and F. Y. Wang. We also discuss special weak versions appropriate to the study of log-concave measures and log-concave perturbations of product measures.  

\end{abstract}
\bigskip


\bigskip

\bigskip

\medskip

\section{Introduction, framework and presentation of the results.}\label{secintro}

Introduced in \cite{BE} the $\Gamma_2$ criterion (also called $CD(\rho,\infty)$ curvature condition) is the best known sufficient condition  for Poincar\'e and log-Sobolev inequalities to hold for some probability measure $\mu$. It reads as $$\Gamma_2(f) \geq \rho \, \Gamma(f)$$ for some $\rho>0$ (see the definitions in the next subsection), i.e. is a pointwise condition. In \cite{ledspin}, M. Ledoux introduced an integrated version $$\mu(\Gamma_2(f)) \geq \rho \, \mu(\Gamma(f))$$ and proved that this integrated version for some $\rho>0$ is equivalent to a Poincar\'e inequality (see Theorem \ref{thmpoinc1} below). The Poincar\'e inequality is thus a mean curvature condition.

As it is well known, Poincar\'e inequality is related to the ``exponential'' concentration of measure, to the $\mathbb L^2(\mu)$ contraction of some associated Markov semi-group (implying exponential stabilization) and to some isoperimetric questions. 
\medskip

During the last years weaker (and also stronger) forms of the Poincar\'e inequality have been discussed. They allow us to describe weaker concentration properties (polynomial for instance) and slower rates of convergence to equilibrium (see subsection \ref{subsecmain}). It is natural to ask whether these weak Poincar\'e inequalities are equivalent to some weak integrated $\Gamma_2$ criteria. This was the starting point of this work. 
\medskip

We then describe some applications of weak integrated $\Gamma_2$ criteria to log-concave measures, perturbation of product measures or of radial measures.
\medskip

\subsection{Framework (The heart of darkness following \cite{BaGLbook}). \\ \\}\label{subsecframe}

We will first introduce the objects we are dealing with. The aficionados of \cite{BaGLbook} will (almost) recognize what is called a full Markov triple therein. Nevertheless in order to understand some of our approaches, one has to understand why this framework is the good one.
 
Let $\mu(dx)= Z_V^{-1} \, e^{-V(x)} \, dx$ be a probability measure defined on an open domain $D \subseteq \mathbb R^n$. When needed, we will require some regularity for $V$ and assume that it takes finite values. We denote by $\mu(f)$ the integral of $f$ w.r.t. $\mu$.
\medskip

If $V$ is in $C^2(D)$, we may introduce the operator $$A=\Delta - \nabla V.\nabla$$ and  the diffusion process $$X^x_t = x + \sqrt 2 \, B_t - \int_0^t \nabla V(X^x_s) ds$$ living in $D$ up to an explosion time $T_\partial^x$ since $\nabla V$ is local Lipschitz. Of course here $B_.$ is a standard Brownian motion. 

When $D=\mathbb R^n$, $T_\partial^x=\sup_{k\in \mathbb N*} T_k^x$ where $T_k^x$ denotes the exit time of the euclidean ball of radius $k$, while if $D$ is a bounded open subset, $T_\partial^x$ denotes the hitting time of the boundary $\partial D$, i.e. $$T_\partial^x = \sup_k T_k^x \quad \textrm{ where } \quad T_k^x = \inf\{t, d(X_t^x,D^c) \leq 1/k\} \, .$$ In the sequel we will assume that  
\begin{equation}\label{eqsortie}
T_\partial^x = +\infty \quad a.s. \textrm{ for all } x\in D \, .
\end{equation}
In other words the process $X_.$ is conservative (in $D$) and we define $P_tf(x)=\mathbb E(f(X_t^x))$ for bounded $f$'s, so that $P_t$ is a markovian semi-group of contractions in $\mathbb L^\infty(D)$.
\medskip

\begin{definition}\label{defH}
We shall say that Assumption (H) is satisfied if \eqref{eqsortie} holds true and if in addition  
\begin{equation}\label{eqsym}
\textrm{$\mu$ is a reversible (symmetric) measure for the process.}
\end{equation}
\end{definition}

We will denote $$\Gamma(f,g)= \langle \nabla f,\nabla g\rangle \quad , \quad  \mathcal E(f,g) = \mu(\Gamma(f,g))$$ the associated Dirichlet form, with domain $\mathcal D(\mathcal E)$. We will write $\Gamma(f)$ for $\Gamma(f,f)$. 

The next result is the key of the construction
\begin{proposition}\label{propESA}
Assume that (H) is satisfied. In the following two cases
\begin{enumerate}
\item \quad $D=\mathbb R^n$,
\item \quad $D$ is an open bounded domain and $V \in C^\infty(D)$,
\end{enumerate}
then $P_t$ extends to a $\mu$-symmetric continuous Markov semi-group on $\mathbb L^2(\mu)$ with generator $\tilde A$ and domain $\mathcal D(\tilde A)$. 

In addition the generator $\tilde A$ is essentially self-adjoint on $C_0^\infty(D)$ ($C^\infty$ functions with compact support). We shall call ESA this property. In particular $C^\infty_0(D)$ is a core for $\mathcal D(\tilde A)$. The latter is exactly the set of $f \in H^2_{loc}(D)$ such that $f$ and $Af$ are in $\mathbb L^2(\mu)$.
\end{proposition}
We shall give a proof of this Proposition in section \ref{secdark}, where sufficient conditions for (H) are discussed as well as examples. For simplicity we will only use the notation $A$ in the sequel both for $A$ and $\tilde A$.
\smallskip

If $g \in \mathcal D(A)$ it holds 
\begin{equation}\label{eqsymgood}
\mathcal E(f,g) = - \, \mu(f \, Ag) \, .
\end{equation}

If $f \in \mathbb L^2(\mu)$ it is well known that $P_tf \in \mathcal D(A)$ for $t>0$ and 
\begin{equation}\label{eqyosi}
\partial_t \, P_tf = A P_tf \, .
\end{equation}
If in addition $f \in \mathcal D(A)$,
\begin{equation}\label{eqyosi1}
\partial_t \, P_tf = A P_tf \, = \, P_t Af \, .
\end{equation}
In particular if $f$ is in $\mathcal D(A)$, for $t>0$, 
\begin{equation}\label{eqyosi2}
\partial_t \, AP_tf = \partial_t \, P_t Af = A \, P_t Af \, .
\end{equation}
\medskip

\subsection{Presentation of the main results. \\ \\}\label{subsecmain}

We define the Poincar\'e constant $C_P(\mu)$ as the smallest constant $C$ satisfying
\begin{equation}\label{eqpoinc}
\Var_\mu(f):=\mu(f^2)-\mu^2(f) \, \leq \, C \, \mu(|\nabla f|^2) \, ,
\end{equation}
for all $f \in C^1_b(D)$ the set of $C^1$ functions which are bounded with a bounded derivative. For simplicity we will say that $\mu$ satisfies a Poincar\'e inequality provided $C_P(\mu)$ is finite. 

As it is well known, the Poincar\'e constant is linked to the exponential stabilization of the Markov semi-group $P_t$.

For a Diffusion Markov Triple, the following is well known (see chapter 4 in \cite{BaGLbook}), it extends to our situation 
\begin{theorem}\label{thmpoinc1}
If (H) is satisfied, the following three statements are equivalent
\begin{enumerate}
\item[(1)] \quad $\mu$ satisfies a Poincar\'e inequality,
\item[(2)] \quad there exists $C$ such that for every $f \in C_0^\infty(D)$ (or $C_b^\infty(D)$ the set of smooth functions with bounded derivatives of any order), it holds 
\begin{equation}\label{eqG2}
\mu(|\nabla f|^2) \, \leq \, C \, \mu((Af)^2) \, ,
\end{equation}
\item[(3)] \quad there exists $C>0$ such that for every $f \in \mathbb L^2(\mu)$, 
\begin{equation}\label{eqexpvar}
\Var_\mu(P_tf) \, \leq \, e^{- \, 2t/C} \, \Var_\mu(f) \, .
\end{equation} 
\end{enumerate}
In addition the optimal constants in \eqref{eqG2} and \eqref{eqexpvar} are equal to $C_P(\mu)$.
\end{theorem}
It is important to check that the previous theorem only requires the properties we have recalled before.  Actually the proof of (1) $\Leftrightarrow$ (3) (\cite{BaGLbook} Theorem  4.2.5) only requires \eqref{eqyosi} so that it is always satisfied. The one  of (2)  $\Leftrightarrow$ (1) (\cite{BaGLbook} Proposition  4.8.3) requires to use ESA. In addition one has to check that the semi-group is ergodic, i.e. that the only invariant functions ($P_tf=f$ for all $t$) are the constants. A proof is provided in the Appendix. 
\medskip

Following D. Bakry we may define (provided $V$ is $C^2$) the $\Gamma_2$ operator 
\begin{equation}\label{eqdefG2}
\Gamma_2(f,g)  \, = \, \frac 12 \, \left[A\Gamma(f,g) - \Gamma(f,Ag) - \Gamma(Af,g)\right] \, .
\end{equation}
for $f,g$ in $C_b^\infty(D)$. A simple calculation yields in this case
\begin{equation}\label{eqdefG2b}
\Gamma_2(f):=\Gamma_2(f,f)  \, = \, \parallel Hess(f)\parallel^2_{HS} + \langle \nabla f, Hess(V) \, \nabla f\rangle \, .
\end{equation}
Using symmetry we get 
\begin{equation}\label{eqG2b}
\mu(\Gamma_2(f,g)) = \mu((Af)(Ag)) \, .
\end{equation}
still for $C^\infty_b$ functions since if (H) is satisfied, they belong to $\mathcal D(A)$. The latter extends to $f,g$ in $\mathcal D(A)$ thanks to ESA.\\

It is important to see that without (H) this result is wrong in general. To justify \eqref{eqG2b} it is at least necessary to know that $\Gamma(f,g) \in \mathcal D(A)$ which is not always the case even for $C_b^\infty$ functions if they are not all in $\mathcal D(A)$, as in the case of reflected diffusions for instance. Fortunately if (H) is satisfied it suffices to verify it for $C_b^\infty$ functions.
\medskip

Assume from now on that \eqref{eqG2b} is satisfied for $f$ and $g$ in the domain of $A$. It immediately follows that, if the curvature-dimension condition $CD(\rho,N)$ i.e. $$\Gamma_2(f) \geq \rho \, |\nabla f|^2 + \frac 1N \, (Af)^2$$ is satisfied, then $$C_P(\mu) \leq \frac{N-1}{\rho \, N}$$ the result being true for $N \in ]1,+\infty]$. This is the famous Bakry-Emery criterion for the Poincar\'e inequality. For $N=+\infty$ the criterion is satisfied provided  $V$ is strictly convex in which case it is also a consequence of Brascamp-Lieb inequality.

The second statement in Theorem \ref{thmpoinc1} is thus sometimes called ``the integrated $\Gamma_2$ criterion''. This statement appears in Proposition 1.3 of M. Ledoux's paper \cite{ledspin}  as ``\textit{a simple instance of the Witten Laplacian approach of Sj\"{o}strand and Helffer}'', but part of the argument goes back to H\"{o}rmander (see e.g. \cite{Alonbast} p.14). It is worth noticing that, if the semi-group does not appear in the statement, it is an essential tool of Ledoux's proof.
\medskip

The integrated $\Gamma_2$ criterion is used in M. Ledoux's work \cite{ledspin} on Gibbs measures. Under the denomination of ``Bochner's method'' it appeared more or less at the same time in the statistical mechanics word. More recently it was used in the context of convex geometry in \cite{Klartuncond,BCO} under the denomination of $\mathbb L^2$ method. Lemma 1 in \cite{BCO} contains another proof (without using the semi-group) of (2) $\Rightarrow$ (1) in the previous Theorem.
\medskip

The third statement in Theorem \ref{thmpoinc1} can be improved in the following way
\begin{proposition}\label{proppoinc1}
The third statement (hence the first two too) of Theorem \ref{thmpoinc1} is equivalent to the following one: there exists $C>0$, such that for every $f$ in a dense subset $\mathcal C$ of $\mathbb L^2(\mu)$ one can find a constant $c(f)$ such that $$\Var_\mu(P_t f) \, \leq \, c(f) \, e^{-2t/C} \, $$ and the optimal $C$ is again $C_P(\mu)$.
\end{proposition}
The proof of this proposition lies on the log-convexity of $t \mapsto \mu(P^2_t f)$ for which several proofs are available (see the simplest one in \cite{CGZ} lemma 2.11 or in \cite{BaGLbook}).

A natural subset $\mathcal C$ is furnished by $\mathbb L^\infty(\mu)$. An exponential decay to $0$ of the variance controlled by the initial uniform norm thus implies that the same holds for the $\mathbb L^2$ norm and is equivalent to the Poincar\'e inequality. 
\medskip

The semi-group property shows that $\mathbb L^2$ decay to $0$ cannot be faster than exponential and the previous result that any uniform decay i.e. $\Var_\mu(P_T f) \leq c \, \Var_\mu(f)$ for some $T>0$, $c<1$ and $f \in \mathbb L^2(\mu)$ implies exponential decay. A natural question is then to describe what happens for slower decays. After a pioneering work by T. Liggett (\cite{Lig}), this question was tackled by M. R\"{o}ckner and F. Y. Wang in \cite{rw}. These authors introduced the notion of weak Poincar\'e inequalities and relate them to all possible decays of the variance along the semi-group. Let us recall the main result in this direction
\begin{theorem}\label{thmWP1}
Consider the following two statements
\begin{enumerate}
\item[(1)] \quad There exists a non-increasing function $\beta_{WP}:(0,+\infty) \to \mathbb R^+$, such that for all $s>0$ and any bounded and Lipschitz function $f$,
\begin{equation}\label{eqWP}
\Var_\mu(f) \, \leq \, \beta_{WP}(s) \, \mu(|\nabla f|^2) \, + \, s \, \Osc^2(f) \, ,
\end{equation}
where $\Osc(f)$ denotes the Oscillation of $f$. \eqref{eqWP} is called a weak Poincar\'e inequality (WPI) and it is clear that we may always choose $\beta_{WP}(s)=1$ for $s\geq 1$.
\item[(2)] \quad There exists a non-increasing function $\xi$ going to $0$ at infinity such that $$\Var_\mu(P_t f) \, \leq \, \xi(t) \, \Osc^2(f) \, .$$
\end{enumerate}
The weak Poincar\'e inequality (1) implies statement (2) with $$\xi(t) = 2 \, \inf\{s>0, \beta_{WP}(s) \, \ln(1/s) \leq 2t \} \, = \, \inf_{s>0} \left(s+ e^{- 2t/\beta(s)}\right).$$ Conversely statement (2) implies statement (1) with $$\beta_{WP}(s) = 2s \, \inf_{r>0} \left(\frac 1r \, \xi^{-1}(r \, \mbox{\rm exp}(1-\frac rs))\right) \, $$ where $\xi^{-1}$ denotes the converse of $\xi$, i.e. $\xi^{-1}(r) = \inf\{s>0, \xi(s)\leq r\}$.
\end{theorem}

\begin{remark}\label{remimprovrw}
R\"{o}ckner and Wang (see \cite{rw} Corollary 2.4 (2)) introduce a trick that allows to improve $\xi$ in the previous result.  The basic idea is to use repeatedly \eqref{eqWP}.  We will choose four sequences: 
\begin{enumerate}
\item[(1)] \quad a decreasing sequence of positive numbers $(\theta_i)_{i\in \mathbb N}$ such that $\theta_0=1$ and $\theta_i \to 0$ as $i \to +\infty$,
\item[(2)] \quad for $i\geq 1$, $\alpha_i =\theta_{i-1} - \theta_i$ so that $\sum_i \alpha_i=1$,
\item[(3)] \quad a sequence $(\gamma_i)_{i\geq 0}$ of positive numbers such that $\gamma_0=1$ and $\prod_i \, \gamma_i=0$,
\item[(4)] \quad for $i\geq 1$, $s_i(t)$ is defined by $e^{-2t \, \alpha_i / \beta_{WP}(s_i(t))}=\gamma_i$, hence $s_i(t)=\beta_{WP}^{-1}(2t \alpha_i/\ln(1/\gamma_i))$. 
\end{enumerate}
Applying \eqref{eqWP} between $t \theta_{i}$ and $t \theta_{i-1}$ we thus have $$\\Var_\mu(P_{\theta_{i-1} t} f) \leq \, e^{-2 \alpha_i t/\beta_{WP}(s_i(t))} \, \Var_\mu(P_{\theta_{i} t} f) \, + \, s_i(t)  \, \Osc^2(f) \, = \,  \gamma_i \, \Var_\mu(P_{\theta_{i} t} f) \, + \, s_i(t)  \, \Osc^2(f) \, ,$$ which yields
\begin{equation}\label{eqWPiter}
\Var_\mu(P_{t} f) \leq \, \sum_{i\geq 0}  (\gamma_i \, s_{i+1}(t))   \, \Osc^2(f) \, .
\end{equation}
So that we may choose $\xi(t) = \sum_{i\geq 0}  (\gamma_i \, s_{i+1}(t))$. \hfill $\diamondsuit$
\end{remark}
\begin{remark}\label{remimprovrw1}
In order to prove that statement (2) implies statement (1) we may follow another route. Using 
\begin{equation}\label{eqvardecay}
\Var_\mu(f) - \Var_\mu(P_t f) \, = \, 2 \, \int_0^t \, \mu(|\nabla P_uf|^2) \, du
\end{equation}
and the fact that $t \mapsto \mu(|\nabla P_tf|^2)$ is non-increasing (we shall recall a proof in the next section), we have $$\Var_\mu(f) \leq 2t \, \mu(|\nabla f|^2) + \Var_\mu(P_t f) \leq 2t  \, \mu(|\nabla f|^2) + \xi(t) \, \Osc^2(f)$$ from which we deduce that $$\beta_{WP}(s) \leq 2 \xi^{-1}(s) \, .$$ This expression is simpler than the one in \cite{rw} we recalled in Theorem \ref{thmWP1}, but can be slightly worse.
 \hfill $\diamondsuit$
\end{remark}
\medskip

\begin{example}\label{exweak}
Let us give some examples of (non optimal) pairs for $(\beta_{WP},\xi)$. 
\begin{enumerate}
\item[(1)] \quad If for $p>0$, $\xi(t) =c' \, t^{-p}$ one can take $\beta_{WP}(s)= c \, s^{-1/p}$. Conversely if $\beta_{WP}(s)= c \, s^{-1/p}$ the previous Theorem yields $\xi(t) =c' \, t^{-p} \, \ln^p(t)$. 

Using the trick in remark \ref{remimprovrw}, when $\beta_{WP}(s) = c s^{-1/p}$, and choosing $\gamma_i=2^{-i}$ and $\alpha_i= \frac{6}{\pi^2} \, i^{-2}$ we get that $\xi(t) \sim t^{-p}$, for large $t$'s, i.e. the logarithmic term disappeared as expected. 
\item[(2)] \quad For $p>0$, $\xi(t)= c' \, \ln^{-p}(1+t)$ and $\beta_{WP}(s)=c \, e^{\delta/s^{1/p}}$. 
\item[(3)] \quad For $0<p<1$, $\xi(t)= c \, e^{- c' t^p}$ and $\beta_{WP}(s)=d' + d \ln^{(1-p)/p}(1+ 1/s)$ .
\end{enumerate}
All the constants depend on $p$. \hfill $\diamondsuit$
\end{example}
\medskip

A natural question is thus to understand whether there is an integrated $\Gamma_2$ version of these weak inequalities or not. This will be done in the next section where we introduce a first weak version: for some decreasing $\beta$ for any bounded $g \in D(A)$ and any $s>0$,
\begin{equation}\label{eqG2w2a}
\textbf{(WI$\Gamma_2$Osc)} \quad \mu(|\nabla g|^2) \, \leq \, \beta(s) \, \mu((Ag)^2) \, + \, s \, \Osc^2(g) \, .
\end{equation}
We shall see that (WI$\Gamma_2$Osc) can be  compared with the weak Poincar\'e inequality.
\medskip

In section \ref{seclogconc} we introduce another, perhaps more natural, weak version 
\begin{equation}\label{eqG2grad2}
\textbf{(WI$\Gamma_2$ grad)}  \quad \mu(|\nabla g|^2) \, \leq \, \beta(s) \, \mu((Ag)^2) \, + \, s \, || |\nabla g|^2||_\infty \, ,
\end{equation} 
which is useful in the log-concave situation, i.e. provided $V$ is convex (not necessarily strictly convex). It is known since S. Bobkov's work \cite{bob99}, that a log-concave probability measure always satisfies some Poincar\'e inequality (see \cite{BBCG} for a direct proof using Lyapunov functions). Recent results by E. Milman (\cite{emil1}) combined with Brascamp-Lieb inequality allow us to get the following result : if $\mu$ is log-concave, $$C_P(\mu) \leq C_{univ} \, \mu(||Hess^{-1} V||_{HS})$$ for some universal $C_{univ}$. We recover this result in corollary \ref{corBL} as a consequence of (WI$\Gamma_2$ grad) (and not Brascamp-Lieb) and obtain new explicit bounds in corollary \ref{corimprovBL} involving $$\mu(\ln^{1+\varepsilon}(1+||Hess^{-1} V||_{HS})$$ only. 
\medskip

The next section \ref{secappli} deals with log-concave perturbations of either log-concave product measures or log-concave radially symmetric measures. Actually M. Ledoux introduced the integrated $\Gamma_2$ criterion in order to study the Poincar\'e inequality of perturbations (non necessarily log-concave but wit a potential whose curvature is bounded from below) of product measures and to obtain results for Gibbs measures on continuous spin systems (\cite{ledspin}). In the same paper he extended his approach to the log-Sobolev constant (see \cite{ledspin} Proposition 1.5 and the comments immediately after  its statement). This approach was then developed in \cite{OR} and several works.

In their subsection 3.4, Barthe and Klartag \cite{BK19} indicate that this method should be used in order to get some results on log-concave perturbations of product measures that are uniformly log-concave in the large, but not for heavy tailed product measures. In section \ref{secappli} we show that the weak integrated $\Gamma_2$ criterion allows us to (partly) recover similar but slightly worse results as in \cite{BK19}. Other results in this direction are shown in \cite{CGperturb}. We then extend the method and replace product measures by radial distributions.
\medskip

\emph{In all the paper, unless explicitly stated, we assume for simplicity that assumption (H) is in force.} 
\bigskip

 \textbf{Dedication.} \textit{A tribute to Michel Ledoux.}

The origin of this work was an attempt to convince M. Ledoux of the interest of weak inequalities of Poincar\'e type. After reading the beautiful wink to Michel's heroes \cite{ledgamma}, we understood that the only way to succeed was to introduce some ``curvature condition'' inside. It was thus natural to weaken the integrated $\Gamma_2$ criterion introduced in \cite{ledspin} and to see what happens. The byproduct results in the paper were a nice surprise. 
\bigskip

\section{Weak integrated $\Gamma_2$.}\label{secweak}

Let us start with an obvious remark: since $\nabla f$ and $Af$ are unchanged when replacing $f$ by $f-a$ for any constant $a$, we have $$\mu(|\nabla f|^2) = \mu(|\nabla (f-a)|^2) = - \mu((f-a) \, A(f-a)) = - \mu((f-a) \, Af) \leq \frac 12 \, \Osc(f) \, \mu(|Af|) \, ,$$
by choosing $a=(\sup(f)+\inf(f))/2$. Using for $s>0$, $2uv \leq \frac 1s \, u^2 + s \, v^2$ we thus deduce, using Cauchy-Schwarz inequality, that for all $s>0$,
\begin{equation}\label{eqG2w1}
\mu(|\nabla f|^2) \, \leq \, \frac{1}{16 s} \, \mu((Af)^2) \, + \, s \, \Osc^2(f) \, .
\end{equation}
This is a special instance of \eqref{eqG2w2a} we recall here: for some decreasing $\beta$ for any bounded $g \in D(A)$ and any $s>0$,
\begin{equation}\label{eqG2w2}
\textbf{(WI$\Gamma_2$Osc)} \quad \mu(|\nabla g|^2) \, \leq \, \beta(s) \, \mu((Ag)^2) \, + \, s \, \Osc^2(g) \, .
\end{equation}
Hence some (very) weak form of the integrated $\Gamma_2$ is always satisfied. The previous inequality is thus certainly insufficient in order to get interesting consequences. 
\begin{remark}\label{remfaible}
Contrary to (WPI), \eqref{eqG2w2} is not always satisfied for $s \geq 1$, so that, apriori, $\beta$ does not necessarily goes to $0$ at infinity. However if \eqref{eqG2w2} is satisfied with two functions $\beta_1$ and $\beta_2$, it is also satisfied with $\beta=\min(\beta_1,\beta_2)$. According to \eqref{eqG2w1}, it is thus always satisfied for $s \mapsto \min(\beta(s),1/16 s)$, so that we may always assume without loss of generality that $\beta$ goes to $0$ at infinity. Again in all what follows we denote $\beta^{-1}(t) = \inf\{s>0 \, , \, \beta(s)\leq t\}$. 
\hfill $\diamondsuit$
\end{remark}

To see how to reinforce \eqref{eqG2w1} it is enough to look at the proof of (2) implies (1) in Theorem \ref{thmpoinc1}. We follow the proof in \cite{ledspin}.

The starting point is again \eqref{eqvardecay},
$$
\Var_\mu(f) - \Var_\mu(P_t f) \, = \, 2 \, \int_0^t \, \mu(|\nabla P_uf|^2) \, du
$$
yielding the equality (1.7) in \cite{ledspin}, $$\Var_\mu(f) = 2 \, \int_0^{+\infty} \, \mu(|\nabla P_tf|^2) \, d\mu$$ as soon as $$\Var_\mu(P_t f) \to 0 \quad \textrm{ as } t \to +\infty \, .$$ Since $\mu$ is symmetric, the latter is satisfied as soon as the semi-group is ergodic, i.e. the eigenspace of $A$ associated to the eigenvalue $0$ is reduced to the constants. Actually this property is ensured by our assumptions: as shown for instance in \cite{rw} Theorem 3.1 and the remark following this theorem, if $\mu(dx)=e^{-V} dx$ is a probability measure with $V$ of $C^1$ class, hence locally bounded, $\mu$ satisfies some weak Poincar\'e inequality so that the above convergence holds true.
\medskip

Now defining $F(t)=\mu(|\nabla P_t f|^2)$, one can check (using \eqref{eqsymgood} and \eqref{eqyosi2}) that 
\begin{equation}\label{eqdecayF}
F'(t)=-2 \mu((AP_t f)^2) \, . 
\end{equation}
Notice that this equality shows that $F$ is non increasing. 
\medskip

Using this property in \eqref{eqvardecay} we also have
\begin{equation}\label{eqdecaygrad}
\mu(|\nabla P_tf|^2) \, \leq \, \frac{1}{2t} \, \Var_\mu(f) \, \leq \, \frac{1}{2t} \, \Osc^2(f) \, .
\end{equation}

Assuming that a weak integrated $\Gamma_2$ inequality \eqref{eqG2w2} is satisfied we get, using that $\Osc(P_tf) \leq \Osc(f)$,  $$ F'(t) \, \leq \, - \, \frac{2}{\beta(s)} \, F(t) \, + \, \frac{2s}{\beta(s)} \, \Osc^2(f) \, .$$ This immediately yields 
\begin{equation}\label{eqdecaygrad1}
\mu(|\nabla P_t f|^2)= F(t) \leq \, e^{-2t/\beta(s)} \, \mu(|\nabla f|^2) \, + \, s \left(1-e^{-2t/\beta(s)}\right) \, \Osc^2(f) \, .
\end{equation}
We may apply the previous inequality replacing $f$ by $P_af$, next $t$ by $t-a$ and use again $\Osc(P_{a}f) \leq \Osc(f)$. Using \eqref{eqdecaygrad} we thus have for $t>a>0$,
\begin{equation}\label{eqdecaygrad2}
\mu(|\nabla P_t f|^2) \leq \inf_{s>0} \left(s+ \frac{1}{2a} \, e^{- 2(t-a)/\beta(s)}\right) \, \Osc^2(f) \, = \, \eta(t) \, \Osc^2(f).
\end{equation}
We have thus obtained
\begin{proposition}\label{propG2-1}
Assume that $\mu$ satisfies a weak integrated $\Gamma_2$ inequality (WI $\Gamma_2$Osc) \eqref{eqG2w2}. Define for $t>a>0$, $$\eta(t) = \inf_{s>0} \left(s+ \frac{1}{2a} \, e^{- 2(t-a)/\beta(s)}\right) \, = \, 2 \inf \{s>0 \, ; \, \beta(s) \ln(1/as) \leq 2(t-a)\} \, .$$ If $\eta$ is integrable at infinity, then for $t>a$, $$\Var_\mu(P_t f) \, \leq \, 2 \, \left(\int_t^{+\infty} \, \eta(u) du\right) \, \Osc^2(f) \, .$$ In particular $\mu$ satisfies a (WPI) where $\beta_{WP}$ is given in Theorem \ref{thmWP1} with $$\xi(t) = 2 \, \left(\int_t^{+\infty} \, \eta(u) du\right).$$
\end{proposition}
\begin{remark}\label{remwG2}
Notice that if $\mu$ satisfies a Poincar\'e inequality we recover the correct exponential decay thanks to proposition \ref{proppoinc1}.

If we come back to \eqref{eqG2w1} we may always use $\beta(s) \sim c/s$. Using proposition \ref{propG2-1} with $a=t/2$  the best possible $\eta(t)$ is of order $1/t$ (for large $t$'s) and thus is not integrable, in accordance with the fact that \eqref{eqG2w1} cannot furnish the rate of decay to $0$ since it is satisfied for all measures $\mu$. 

Notice that, as for the (WPI), if $\beta(s) = c s^{- 1/(p+1)}$ we obtain $\eta(t) = \, c' \,  (\ln(t)/t)^{p+1}$ and finally $\xi(t) \sim c' \, (\ln(t)/t)^p$. But here again we may apply the trick of remark \ref{remimprovrw}, simply replacing \eqref{eqWP} by \eqref{eqdecaygrad1}, yielding
\begin{equation}\label{eqdecaygradrw}
\mu(|\nabla P_{t} f|^2) \leq \, \sum_{i\geq 0}  (\gamma_i \, s_{i+1}(t))   \, \Osc^2(f) \, ,
\end{equation}
with $$s_i(t) = \beta^{-1}(2t\alpha_i/\ln(1/\gamma_i)) \, .$$ As for (WPI) this remark allows us to skip the logarithmic term. 
\hfill $\diamondsuit$
\end{remark}
\begin{remark}\label{remwG2b}
Taking $a=\mu(f)$ we may replace \eqref{eqG2w1} by $$\mu(|\nabla f|^2) \leq \frac{1}{4s} \, \mu((Af)^2) \, + \, s \, \Var_\mu(f) \, ,$$ so that we could also consider weak inequalities of the form
\begin{equation}\label{eqwpvar}
\mu(|\nabla f|^2) \leq \beta(s) \, \mu((Af)^2) \, + \, s \, \Var_\mu(f) \, .
\end{equation} 
It is immediately seen that the previous derivation is unchanged if we replace $\Osc^2(f)$ by $\Var_\mu(f)$ so that if $\eta$ is integrable we get $$\Var_\mu(P_t f) \, \leq \, 2 \, \left(\int_t^{+\infty} \, \eta(u) du\right) \, \Var_\mu(f) \, .$$ But according to what we already said, such a decay implies that $\mu$ satisfies a Poincar\'e inequality, hence thanks to Theorem \ref{thmpoinc1} that $\beta$ is constant equal to $C_P(\mu)$ (or if one prefers that $\beta(0) < +\infty$). Thus, in the other cases, \eqref{eqwpvar} furnishes a non-integrable $\eta$. \hfill $\diamondsuit$
\end{remark}
\medskip

Let us look at the converse statement. According to Theorem \ref{thmWP1} we may associate some (WPI) inequality to any decay controlled by the Oscillation. Thus for $a=\mu(f)$,
\begin{eqnarray}\label{eqback1}
\mu^2(|\nabla f|^2) \, &=& \, -\mu^2((f-a)Af) \, \leq \, \mu((Af)^2) \, \Var_\mu(f) \, \nonumber\\ &\leq& \, \mu((Af)^2) \, \left(\beta_{WP}(s) \, \mu(|\nabla f|^2) + s \, \Osc^2(f)\right) \, .
\end{eqnarray} 
Since $u^2 \leq Cu + B$ implies that $$u \leq \frac 12 \, \left(C + (C^2+4B)^{\frac 12}\right) \leq C+B^{\frac 12} \, ,$$ we obtain 
\begin{eqnarray*}
\mu(|\nabla f|^2) &\leq&  \beta_{WP}(s) \, \mu((Af)^2)) + s^{\frac 12} \, \mu^{\frac 12}((Af)^2) \, \Osc(f) \, , \\ &\leq& (\beta_{WP}(s)+ \frac 12) \, \mu((Af)^2))+\frac 12 \, s \, \Osc^2(f) \, .
\end{eqnarray*} 
We have thus obtained (since we know that $\mu$ always satisfies some (WPI) inequality) and according to remark \ref{remfaible}
\begin{proposition}\label{propG2-2}
$\mu$ always satisfies a weak integrated $\Gamma_2$ inequality (WI $\Gamma_2$ Osc) \eqref{eqG2w2}, with $$\beta(s) = \min\left(1/2 + \beta_{WP}(2s) \, , \, \frac{1}{16 s}\right) \, .$$ 
\end{proposition}
\medskip

The previous results need some comments. 

In first place, if we cannot assume that $\beta(s)=1$ for $s \geq 1$ in the weak integrated $\Gamma_2$ inequality \eqref{eqG2w2}, the interesting behaviour of this function is nevertheless as $s \to 0$ for proposition \ref{propG2-1} to have some interest.

In second place proposition \ref{propG2-2} is certainly non sharp. In particular we do not recover the same $\beta$ when $\beta_{WP}$ is constant, i.e. when $\mu$ satisfies a Poincar\'e inequality, while using \eqref{eqback1} with $s=0$ yields the correct value.

A still worse remark is that the previous proposition cannot be always used in conjunction with proposition \ref{propG2-1}. Indeed if $\beta_{WP}(s) \geq c/s$ as it is the case in the second case of example \ref{exweak} the $\eta$ obtained in proposition \ref{propG2-1} is not integrable. 
\medskip

Let us look at some other example.
\begin{example}\label{exweak2}
Assume that for some $p>0$, $\beta_{WP}(s)= cs^{- 1/p}$. In this case one can improve upon the result of proposition \ref{propG2-2}. Indeed we may replace the weak Poincar\'e inequality by its equivalent Nash type inequality $$\Var_\mu(f) \, \leq \, c \, \left(p + (1/p)^p\right)^{\frac{1}{p+1}} \, \mu^{\frac{p}{p+1}}(|\nabla f|^2) \, \Osc^{\frac{2}{p+1}}(f) \, .$$ We thus deduce $$\mu^2(|\nabla f|^2) \, \leq \, \mu((Af)^2) \, \Var_\mu(f) \leq c(p) \, \mu((Af)^2) \, \mu^{\frac{p}{p+1}}(|\nabla f|^2) \, \Osc^{\frac{2}{p+1}}(f)$$ for some $c(p)$ that may change from line to line, so that $$\mu^2(|\nabla f|^2) \, \leq \, c(p) \, \mu^{\frac{p+1}{p+2}}((Af)^2) \, \Osc^{\frac{2}{p+2}}(f)$$ and finally that $\mu$ satisfies a weak integrated $\Gamma_2$ inequality with $$\beta(s) = c(p) \, s^{-1/(p+1)} \, .$$ This result is of course better than the $s^{-1/p}$ obtained by directly using proposition \ref{propG2-2} and according to remark \ref{remimprovrw} allows to recover the correct decay for $\xi(t)$.
\hfill $\diamondsuit$
\end{example}
\medskip

\section{The log-concave case.}\label{seclogconc}

If one wants to mimic (WPI) it seems more natural to consider another type of weak integrated $\Gamma_2$ inequalities, namely
\begin{equation}\label{eqwG2grad}
\textbf{(WI $\Gamma_2$ grad)} \quad \mu(|\nabla g|^2) \, \leq \, \beta(s) \, \mu((Ag)^2) \, + \, s \, |||\nabla g|^2||_{\infty} \, .
\end{equation}
But contrary to the previous derivation it is no more true that $|||\nabla P_t f|^2||_{\infty} \leq |||\nabla f|^2||_{\infty}$ so that the analogue of \eqref{eqdecaygrad} will involve $\sup_{u\leq t} |||\nabla P_uf|^2||_{\infty}$ which is not really tractable. 
\medskip

If we want to guarantee $|||\nabla P_t f|^2||_{\infty} \leq |||\nabla f|^2||_{\infty}$ a sufficient condition is that $\mu$ is log-concave, i.e. $V$ is convex. Indeed in this case on can show (see a stochastic immediate proof in \cite{CGsemin}) that 
\begin{equation}\label{eqcommut}
|\nabla P_t f|^2 \leq P^2_t(|\nabla f|) \leq P_t(|\nabla f|^2) \leq |||\nabla f|^2||_\infty \, .
\end{equation}
In this case we will thus obtain the analogue of \eqref{eqdecaygrad1} 
\begin{equation}\label{eqdecaylip}
 \mu(|\nabla P_t f|^2) \leq e^{-2t/\beta(s)} \, \mu(|\nabla f|^2) + \, s \, |||\nabla f|^2||_\infty \, .
\end{equation}
The difference with the previous section is that \eqref{eqwG2grad} is satisfied by $\beta(s)=0$ for $s\geq 1$. The converse function $\beta^{-1}$ is thus bounded by $1$, hence integrable at the origin.

Now we can use the trick described in Remark \ref{remimprovrw} which yields, 
\begin{equation}\label{eqtricklc}
\mu(|\nabla P_t f|^2) \leq \left(\sum_{i=0}^{+\infty} \, \gamma_i \, \beta^{-1}(2 \alpha_{i+1} \, t/\ln(1/\gamma_{i+1}))\right) \, |||\nabla f|^2||_\infty = \eta(t) \,   \, |||\nabla f|^2||_\infty \, . 
\end{equation}
Since $\beta^{-1}$ is integrable at $0$, we have thus obtained after a simple change of variable, provided $\eta$ is integrable at infinity
\begin{eqnarray}\label{eqmil}
\Var_\mu(f) \, &\leq& \, 2 \, \left(\int_0^{+\infty} \, \eta(u) du\right) \, |||\nabla f|^2||_\infty \, \nonumber \\ &\leq&  \, \left(\sum_{i=0}^{+\infty} \, \frac{\gamma_i \, \ln(1/\gamma_{i+1})}{\alpha_{i+1}} \right) \, \left(\int_0^{+\infty} \, \beta^{-1}(t) dt\right) \,   |||\nabla f|^2||_\infty  \, . 
\end{eqnarray}
As before we may choose $\gamma_i=2^{-i}$ and $\alpha_i=\frac{6}{\pi^2} i^{-2}$ so that $$\sum_{i=0}^{+\infty} \, \frac{\gamma_i \, \ln(1/\gamma_{i+1})}{\alpha_{i+1}} = \kappa$$ where $\kappa$ is thus a universal constant. Hence if $\beta^{-1}$ is integrable with integral equal to $M_\beta$ we have obtained 
\begin{equation}\label{eqmil2}
\Var_\mu(f) \leq \kappa \, M_\beta \, |||\nabla f|^2||_\infty  \, .
\end{equation}
As first shown by E. Milman in \cite{emil1}, for log-concave measures \eqref{eqmil2} implies a Poincar\'e inequality. A semi-group proof of E. Milman's result was then given by M. Ledoux in \cite{ledlogconc}. Another semi-group proof and various improvements were recently shown in \cite{CGlogconc}. We shall follow the latter to give a precise result.

Starting with $$\mu(|f-\mu(f)|) \leq \Var_\mu^{1/2}(f) \leq \kappa^{1/2} \, M_\beta^{\frac 12} \, |||\nabla f|||_\infty \, ,$$ we deduce from \cite{CGlogconc} Theorem 2.7 that the $\mathbb L^1$ Poincar\'e constant $C'_C(\mu)$ is less than $16 \, \sqrt{\kappa \, M_\beta}/\pi$. Using Cheeger's inequality $C_P(\mu) \leq 4 \, (C'_C(\mu))^2$ we have thus obtained
\begin{proposition}\label{propwG2logconc}
Assume that $\mu$ is log-concave and satisfies a weak integrated $\Gamma_2$ inequality (WI $\Gamma_2$ grad) \eqref{eqwG2grad}. 
Then $$C_P(\mu) \, \leq \, \frac{1024}{
\pi^2} \, \kappa \, M_\beta \, ,$$ where $\kappa$ is some (explicit) universal constant and $M_\beta=\int_0^{+\infty} \, \beta^{-1}(t) dt$.
\end{proposition}
\medskip

It turns out that there always exists a (non necessarily optimal) function $\beta$ such that \eqref{eqwG2grad} is satisfied for a log-concave measure $\mu$

Indeed recall \eqref{eqdefG2} and \eqref{eqG2b}. We have
\begin{eqnarray*}
\mu((Af)^2) &=& \mu(||Hess f ||^2_{HS}) + \mu(\langle \nabla f,Hess V \, \nabla f\rangle)\\ &\geq& \frac 1u \, \mu(|\nabla f|^2 \, \mathbf 1_{||Hess V||_{HS}\geq \frac 1u}) \\ &\geq& \frac 1u \, \mu(|\nabla f|^2) \, - \, \frac 1u \, \mu(\{||Hess V||_{HS}\leq \frac 1u\}) \, |||\nabla f|^2||_\infty \, .
\end{eqnarray*}
It follows $$\mu(|\nabla f|^2) \leq \beta(s) \, \mu((Af)^2) + s \, |||\nabla f|^2||_\infty $$ with 
\begin{equation}\label{eqBL}
\beta^{-1}(s) = \mu(||Hess^{-1}V||_{HS} \geq s) \, .
\end{equation}
Since $$\mu(||Hess^{-1}V||_{HS}) = \int_0^{+\infty} \, \mu(||Hess^{-1}V||_{HS} \geq s) \, ds$$ we have obtained
\begin{corollary}\label{corBL}
If $\mu$ is log-concave and such that $\mu(||Hess^{-1}V||_{HS}) <+\infty$, then $$C_P(\mu) \leq C_{univ} \, \mu(||Hess^{-1}V||_{HS}) \, ,$$ for some universal constant $C_{univ}$.
\end{corollary}
This result is not new and as remarked by E. Milman is an immediate consequence of the fact that \eqref{eqmil2} implies that $\mu$ satisfies some Poincar\'e inequality and of one of the favorite inequality of M. Ledoux, namely the Brascamp-Lieb inequality $$\Var_\mu(f) \leq \mu(\langle \nabla f, Hess^{-1}V \, \nabla f\rangle) \, \leq \, \mu(||Hess^{-1}V||_{HS}) \; |||\nabla f|^2||_{\infty} \, .$$ Actually this method furnishes a slightly better pre-constant than the one obtained with our method (since our $\kappa \geq 1$).
\medskip

Still in the log-concave situation, if we assume \eqref{eqG2w2} we may derive another control for the Poincar\'e constant. 
\begin{proposition}\label{propwG2logconc1}
Assume that $\mu$ is log-concave and satisfies a weak integrated $\Gamma_2$ inequality (WI $\Gamma_2$Osc) \eqref{eqG2w2}. If in addition there exists a function $s(t)$ such that 
\begin{equation}\label{eqdecays}
\int_0^{+\infty} s(t) \, dt = \frac{s_0}{2} < \frac{1}{12} \quad \textrm{ and } \quad \int_0^{+\infty} \, e^{-2t/\beta(s(t))} \, dt \, = \, \kappa/2 < +\infty \, ,$$ then $$C_P(\mu) \leq \frac{64  \ln(2) \, \kappa}{(1-6s_0)^2} \, .
\end{equation}
\end{proposition}
\begin{proof}
Starting with \eqref{eqdecaygrad1} in the simplified form $$\mu(|\nabla P_t f|^2)= F(t) \leq \, e^{-2t/\beta(s(t))} \, \mu(|\nabla f|^2) \, + \, s(t)  \, \Osc^2(f) \, ,$$ we get $$\Var_\mu(f) \leq \kappa \, \mu(|\nabla f|^2) + s_0 \, \Osc^2 f$$ so that the conclusion follows from \cite{CGlogconc} Theorem 9.2.14. 
\end{proof}
\medskip

Still in the log-concave case it was shown by M. Ledoux in \cite{ledgap} that $$|||\nabla P_t f|||_\infty \leq \frac{1}{\sqrt{2t}} \, || f||_\infty$$ so that replacing $f$ by $f-a$ with $a=\frac 12 (\inf f + \sup f)$ we have $$|||\nabla P_t f|||_\infty \leq \frac{1}{2 \sqrt{2 t}} \Osc(f) \, .$$ This bound was improved in \cite{CGsemin} replacing $\sqrt 2$ by $\sqrt \pi$ and is one of the key element in the proof of  Theorem 2.7 in \cite{CGlogconc}. 

We may combine this bound with the (WI $\Gamma_2$grad) inequality in order to improve upon the previous result. If a (WI $\Gamma_2$grad) inequality is satisfied we have
\begin{eqnarray*}
\mu(|\nabla P_{2t}f|^2) &\leq&  e^{-2t/\beta(s)} \, \mu(|\nabla P_tf|^2) + s \, |||\nabla P_tf|^2||_\infty \\ &\leq& e^{-2t/\beta(s)} \, \mu(|\nabla f|^2) + \frac{s}{4 \, \pi t} \, \Osc^2(f) \, .
\end{eqnarray*}
We have thus obtained 
\begin{proposition}\label{propwG2logconc2}
Assume that $\mu$ is log-concave and satisfies a weak integrated $\Gamma_2$ inequality (WI $\Gamma_2$ grad) \eqref{eqwG2grad}. If in addition there exists a function $s(t)$ such that 
\begin{equation}\label{eqdecays2}
\int_0^{+\infty} \frac{s(t)}{4\pi t} \, dt = \frac{s_0}{4} < \frac{1}{24} \quad \textrm{ and } \quad \int_0^{+\infty} \, e^{-2t/\beta(s(t))} \, dt \, = \, \kappa/4 < +\infty \, ,$$ then $$C_P(\mu) \leq \frac{64  \ln(2) \, \kappa}{(1-6s_0)^2} \, .
\end{equation}
\end{proposition}
In the previous proposition we can choose a generic function $s(t)$ given by
\begin{equation}\label{eqchoixs}
s(t) = \frac{\theta}{16} \, \left(t \, \mathbf 1_{t\leq 2} \, + \, \ln^{-(1+\theta)}(t) \, \mathbf 1_{t>2}\right) \, , 
\end{equation}
so that $$\int_0^{+\infty} \frac{s(t)}{4\pi \, t} \, dt \, = \, \frac{\theta}{32\pi} + \frac{1}{64 \pi \, \ln^{\theta}(2)} \, \leq \, \frac{1}{48} $$ as soon as $0<\theta \leq 1$. So we may always choose
\begin{equation}\label{eqchoixconst}
\kappa = 4 \, \left(2 + \int_2^{+\infty} \, e^{-2t/\beta((\theta/16) \, \ln^{-(1+\theta)}(t))} dt\right) \; , \; s_0 = \frac{1}{12} \; , \;  C_P(\mu) \leq 256 \, \ln(2) \, \kappa \, . 
\end{equation}

As we previously saw, we may also use the previous proposition with $$\beta^{-1}(s)=\mu(||Hess^{-1}V||_{HS}\geq s) \, .$$ This yields

\begin{corollary}\label{corimprovBL}
If $\mu$ is log-concave and such that $M_\varepsilon:=\mu(\ln^{1+\varepsilon}(1+||Hess^{-1}V||_{HS}))<+\infty$ for some $\varepsilon >0$, then $$C_P(\mu) \leq c + 4 \, \max\left(2, \exp \left(\left[\frac{2^\varepsilon \, 64 \, M_\varepsilon}{\theta}\right]^{\frac{1}{\varepsilon - \theta}} \, \right) \right) \, ,$$ with $\theta=1$ if $\varepsilon \geq 2$ and $\theta=\varepsilon/2 $ if $\varepsilon \leq 2$, where $c$ is some universal constant. 
\end{corollary}
\begin{proof}
Denote by $M_\varepsilon=\mu(\ln^{1+\varepsilon}(1+||Hess^{-1}V||_{HS}))$. According to Markov inequality $$\beta^{-1}(s) \leq \frac{M_\varepsilon}{\ln^{1+\varepsilon}(1+s)} \, .$$ 
It follows $$\beta(t) \leq \exp \left[\left(\frac{M_\varepsilon}{t}\right)^{\frac{1}{1+\varepsilon}}\right]$$ so that for $t\geq 2$,$$\beta(s(t))\leq \exp \left[\left(\frac{8M_\varepsilon}{\theta}\right)^{\frac{1}{1+\varepsilon}} \, \ln^{\frac{1+\theta}{1+\varepsilon}}(1+t)\right] \, .$$ 
In particular, using $t^2\geq t+1$ for $t\geq 2$, $$\frac 12 \, \ln(t) \geq \left(\frac{8M_\varepsilon}{\theta}\right)^{\frac{1}{1+\varepsilon}} \, \ln^{\frac{1+\theta}{1+\varepsilon}}(1+t)$$ as soon as $$t \geq \max\left(2, \exp \left[\frac{2^\varepsilon \, 64 \, M_\varepsilon}{\theta}\right]^{\frac{1}{\varepsilon - \theta}}\right) \, .$$ For such $t$'s we thus have $$e^{-2t/\beta(s(t))} = e^{-2 \exp(\ln(t) - \ln(\beta(s(t))))} \leq e^{-2 \sqrt t}$$ so that finally $$\frac{\kappa}{4} \leq \max\left(2, \exp \left[\frac{2^\varepsilon \, 64 \, M_\varepsilon}{\theta}\right]^{\frac{1}{\varepsilon - \theta}}\right) + \int_2^{+\infty} e^{-2\sqrt t} \, dt \, .$$
Hence the result choosing $\theta=1$ if $\varepsilon \geq 2$ and $\theta=\varepsilon/2$ otherwise.
\end{proof}
Of course our bounds are far from being sharp. Notice that the previous corollary allows to look at Subbotin distributions $\mu(dx)= Z^{-1} e^{-|x|^p} \, dx$ for large $p$'s, while Brascamp-Lieb inequality cannot be used. However other known methods (see e.g. S. Bobkov's results on radial measures in \cite{bobsphere}) furnish better bounds in this case. Of course the previous corollary covers non radial cases.

\begin{remark}\label{rempower}
If $M_\varepsilon := \mu(||Hess^{-1}V||_{HS}^\varepsilon) < + \infty$ for some $\varepsilon >0$ we can obtain another explicit bound choosing $\theta=1$ in \eqref{eqchoixs}. Using again Markov inequality we have $\beta(s) \leq (M_\varepsilon/s)^{1/\varepsilon}$ so that 
\begin{eqnarray*}
\int_2^{+\infty}  \, e^{-2t/\beta(s(t))} \, dt &\leq& \int_2^{+\infty} \, e^{-2t/\ln^{2/\varepsilon}(M_\varepsilon^{1/\varepsilon} \, t)} dt \\ &\leq& \frac 12 \, M_\varepsilon^{1/\varepsilon} \, \ln^{2/\varepsilon}(M^{2/\varepsilon}_\varepsilon) + \int_2^{+\infty} \, e^{-2t/\ln^{2\varepsilon}(t^2)} dt
\end{eqnarray*}
and finally $$C_P(\mu) \leq c(\varepsilon) + \max\left(2, \frac 12 \, M_\varepsilon^{1/\varepsilon} \, \ln^{2/\varepsilon}(M^{2/\varepsilon}_\varepsilon)\right) \, .$$ Notice that for $\varepsilon=1$ we recover a slightly worse result than corollary \ref{corBL} since an extra logarithm appears. Of course choosing $s(t)$ with a slower decay, we may improve upon this result but it seems that in all cases an extra worse term always appears. In addition constants are quite bad. But of course the result is new for $\varepsilon<1$.
\hfill $\diamondsuit$
\end{remark}
\medskip

\section{Some applications: perturbation of product measures and radial measures.}\label{secappli}

We will first recall how the (usual) integrated $\Gamma_2$ criterion can be used in order to relate the Poincar\'e constant of $\mu$ to the ones of its one dimensional conditional distributions, in some special situations. We copy here Proposition (3.1) in \cite{ledspin} and its proof to see how to potentially extend it. In the sequel we denote $$SG(\mu) = \frac{1}{C_P(\mu)}$$ the spectral gap of $\mu$.

\begin{proposition}{\textit{(M. Ledoux)}}\label{propledspin}

Let $\mu(dx)=Z^{-1} \, e^{-W(x) -\sum_{i=1}^n \, h_i(x_i)} \, dx=Z^{-1} \, e^{-V(x)} dx$ be a probability measure on $\mathbb R^n$, $W$ and the $h_i$'s being $C^2$. Introduce the one dimensional conditional distributions $$\eta_{i,x}(dt) = Z_{i,x}^{-1} \, e^{-W(x_1,...,x_{i-1},t,x_i,..x_n) - h_i(t)} \, dt \, .$$ Let $$S = \inf_{i,x} SG(\eta_{i,x}) \, .$$ Assume that $Hess W(x) \geq w$ and $\max_i \, \partial^2_{ii} W(x) \leq \bar w$ for all $x \in \mathbb R^n$. 

Then $$SG(\mu) \geq S + w - \bar w \, .$$
\end{proposition}
\begin{proof}
It holds
\begin{eqnarray}\label{eqproof}
\Gamma_2 f &=& \sum_{i,j} (\partial^2_{ij} f)^2 + \sum_i h''_i(x_i) (\partial_i f)^2 + \langle \nabla f , Hess W \, \nabla f\rangle \nonumber \\ &\geq& \sum_{i} (\partial^2_{ii} f)^2 + \sum_i h''_i(x_i) (\partial_i f)^2 + w |\nabla f|^2 \\ &\geq& \sum_{i} (\partial^2_{ii} f)^2 + \sum_i (h''_i(x_i) + \partial^2_{ii} W) \, (\partial_i f)^2 + (w - \bar w) \, |\nabla f|^2 \, \nonumber \\ &\geq&  \sum_{i} \Gamma_{2,i} f + (w - \bar w) \, |\nabla f|^2 \, . \nonumber
\end{eqnarray}
It follows
\begin{eqnarray}\label{eqproof2}
\mu((Af)^2) = \mu(\Gamma_2 f) &\geq& \sum_{i} \mu(SG(\eta_{i,x}) \, |\partial_i f|^2) + (w - \bar w) \, \mu(|\nabla f|^2) \, \\ &\geq& (S+w - \bar w) \, \mu(|\nabla f|^2) \, , \nonumber
\end{eqnarray}
hence the result applying Theorem \ref{thmpoinc1}.
\end{proof}

\begin{remark}\label{remtensor}
Choosing $W=0$ the previous result contains the renowned tensorization property of Poincar\'e inequality $$C_P(\otimes_i \, \mu_i) \leq \max_i C_P(\mu_i) \, .$$ Similar results for weak Poincar\'e inequalities involve a ``dimension dependence'' (see e.g. \cite{BCR2}).
\hfill $\diamondsuit$
\end{remark}

\begin{remark}\label{remassumpgibbs}
For the proof of proposition \ref{propledspin} to be rigorous, it is enough to assume that ESA is satisfied for $C_0^\infty(\mathbb R^n)$ (which is implicit in M. Ledoux's work). Indeed in this case one only has to consider such test functions. The delicate point in the previous proof is that one has to check $$\mu(\Gamma_{2,i} f)=\mu((A_i f)^2)$$ where $A_if = \partial^2_{ii} f - (h'(x_i)+\partial_i W)\partial_i f$ in order to use the integrated $\Gamma_2$ criterion. If $f$ is compactly supported, this is immediate as we already discussed in the introduction. Hence for $D=\mathbb R^n$, (H) ensures that the result holds true. 

The case of a bounded domain $D$ will be discussed later. \hfill $\diamondsuit$
\end{remark}
\medskip

In the previous proof, assume that $w=0$ ($W$ is convex), we thus obtain $$
\Gamma_2 f \geq \sum_i \, \mu(h_i''(x_i) \, (\partial_i f)^2)$$ so that, as we did for obtaining \eqref{eqBL} we have for $u>0$, since we may integrate w.r.t. $\mu$,
\begin{equation}\label{eqproofh}
\mu(|\nabla f|^2) \, \leq \, u \, \mu((Af)^2) \, + \,  \mu\left(\min_i (h_i''(x_i) \leq 1/u)\right) \, |||\nabla f|^2||_\infty \, 
\end{equation}
that furnishes a (WI $\Gamma_2$ grad) inequality. Of course $$\mu\left(\min_i (h_i''(x_i) \leq  1/u)\right) \leq n \, \max_i \mu\left(h_i''(x_i) \leq \frac 1u\right) \, .$$ 
We have seen that such a weak inequality is interesting provided on one hand $\mu$ is log-concave and on the other hand $u \mapsto \max_i \mu\left(h_i''(x_i) \leq \frac 1u\right)$ which is clearly non-increasing goes to $0$ as $u \to +\infty$. We will thus assume that all $h_i$ are convex, yielding thanks to Proposition \ref{propwG2logconc2} with the choice \eqref{eqchoixs} with $\theta=1$
\begin{lemma}\label{lemconcent}
Let $\mu(dx)=Z^{-1} \, e^{-W(x) -\sum_{i=1}^n \, h_i(x_i)} \, dx=Z^{-1} \, e^{-V(x)} dx$ be a probability measure on $\mathbb R^n$, $W$ and the $h_i$'s being convex and $C^2$. Define $$\alpha(v) = \max_i \, \mu(h''_i(x_i) \leq v)$$ and assume that (the non-decreasing) $\alpha$ goes to $0$ as $v \to 0$. Then $$C_P(\mu) \leq 256 \, \ln(2) \, \kappa$$ with $$\kappa = 4\left(2 +\int_2^{+\infty} e^{-2t/\alpha^{-1}(1/16 \, n \, \ln^2(t))} \, dt \right) \, .$$
\end{lemma}

Let us illustrate this situation in the particular case $h_i(u) =  |u|^p$ for $p>1$. We immediately see that the situation is completely different depending on whether $p<2$ or $p>2$. Denote by $\mu_i$ the probability distribution of $x_i$ under $\mu$. For $p<2$ we have to control the tails of $\mu_i$ while for $p>2$ we have to control the mass of small intervals centered at the origin. 
\begin{remark}\label{remtrue}
For $p<2$, $h_p:u \mapsto |u|^p$ is not $C^2$. But if $p>1$, the only problem lies at the origin, and using that $h''_p$ is integrable at the origin it is not difficult to check (regularizing $h_p$ at the origin for instance) that all what was done above is still true.
\hfill $\diamondsuit$
\end{remark}
More generally we may consider $h_i$'s who satisfy similar concentration bounds. Let us state a first result

\begin{proposition}\label{thmBK}
Let $\mu(dx)= Z^{-1} \, e^{-W(x) \, - \sum_{i=1}^n \, h_i(x_i)} \, dx$ be a probability measure on $\mathbb R^n$. We assume that the $h_i$'s are even convex functions. In addition we assume that for all $i=1,...,n$, $$h''_i(u) \geq \rho(|u|)$$ where $\rho$ is a non-increasing positive function going to $0$ at infinity. Then for all even convex function $W$ it holds $$C_P(\mu) \leq 4 \left(2 + \int_2^{+\infty} \, e^{-2t \, \rho(\sqrt{2 \max_i C_P(\eta^i)} \, \ln(n \, \ln^2(t))} \, dt\right) \, ,$$ where $\eta^i(du) = Z_i^{-1} \, e^{-h_i(u)} \, du$.
\end{proposition}
\medskip
\begin{proof}
According to Prekopa-Leindler theorem we know that the $i$-th marginal law $\mu_i$ of $\mu$, i.e. the $\mu$ distribution of $x_i$, is a one dimensional distribution, that can be written 
\begin{equation}\label{eq1d}
\mu_i(du) = Z_i^{-1} \, \rho_i(u) \, e^{ - h_i(u)} \, du \, ,
\end{equation}
with an even and log-concave (thus non increasing on $\mathbb R^+$) function $\rho_i$. For such one dimensional distributions we may use a remarkable result due to O. Roustant, F. Barthe and B. Ioos (see \cite{barrous}) recalled in proposition 6 of \cite{BK19}, namely
\begin{lemma}{\textit{(Roustant-Barthe-Ioos)}}\label{lemRBI}

Let $\eta(du) = e^{-V(u)} \, \mathbf 1_{(-b,b)}(u) \, du$ be a probability measure on $\mathbb R$, with $V$ a continuous and even function. For any even function $\rho$ which is non-increasing on $\mathbb R^+$ and such that $\nu(du) = \rho(u) \, \eta(du)$ is a probability measure, it holds $$C_P(\nu) \, \leq \, C_P(\eta) \, .$$
\end{lemma}
\medskip

Applying the lemma we get 
\begin{equation}\label{eqroust}
C_P(\mu_i) \leq C_P(Z^{-1} \, e^{- h_i(u)} \, du) \, := \, C_P(\eta^i) \, .
\end{equation}
We can thus use the concentration of measure property obtained via the Poincar\'e inequality, first shown by S. Bobkov and M. Ledoux (\cite{bobled}). Here we use an explicit form we found in \cite{BaGLbook} (4.4.6). Since $u \mapsto u$ is $1$-Lipschitz and centered (again thanks to symmetry), it yields
\begin{equation}\label{eqconctroy}
\mu(h''(x_i) \leq 1/u) \leq \mu(|x_i|\geq \rho^{-1}(1/u)) \leq 6 \, \exp \left( - \, \frac{\rho^{-1}(1/u)}{\sqrt{C_P(\eta^i)}}\right) \, .
\end{equation}
We thus have for $v>0$
\begin{equation}\label{eqproof5}
v \, \mu((Af)^2) \, + \,  6 \, n  \, \max_i \exp \left( - \, \frac{\rho^{-1}(1/v)}{\sqrt{C_P(\eta^i)}}\right)  \, ||\nabla f|^2||_\infty \, \geq \, \mu(|\nabla f|^2) \, ,
\end{equation}
yielding, for $s>0$ small enough,
\begin{equation}\label{eqproof6}
\beta(s) = \frac{1}{\rho \left( \sqrt{\max_i C_P(\eta^i)} \, \ln(6 n/s) \right)} \, .
\end{equation}
It remains to use the lemma \ref{lemconcent}.
\end{proof}
 
\begin{remark}\label{remBK}

When $h_i(u) = |u|^p$ for some $1<p\leq 2$, one knows that $C_P(\eta^i) \leq \frac{4}{p^{2(1-1/p)}}$ according to \cite{BJMsubbot} Theorem 2.1. It follows that for some (explicit) constant $c(p)$,
$$C_P(\mu) \leq \, \frac{c(p)}{p(p-1)} \, (1+\ln^{2-p}(6 n)) \, .$$ 
\smallskip

The study of such $\mu$'s is not new.  A much better result has been recently shown by F. Barthe and B. Klartag (see Theorem 1 in \cite{BK19}), 
\begin{theorem}{\textit{(Barthe-Klartag)}}\label{thmbarklar}
Let $\mu(dx)= Z^{-1} \, e^{-W(x) \, - \sum_{i=1}^n \, |x_i|^p} \, dx$ be a probability measure. We assume that $1 \leq p \leq 2$ and that $W$ is an even convex function. Then
$$C_P(\mu) \leq C \, \ln^{\frac{2-p}{p}}(\max(n,2)) \,  ,$$ where $C$ is some universal constant.
\end{theorem}
The key point here is naturally that the result holds true for any even and convex $W$. The proof by Barthe and Klartag lies on a lot of properties of log-concave measures and uses in particular the extension of the gaussian correlation inequality shown by Royen, to mixtures of gaussian measures. We of course refer the reader to \cite{BK19}. We do not only loose something on the power of the logarithm, but the constant becomes infinite as $p$ goes to $1$, which is natural since the $\Gamma_2$ requires some strict convexity except at some point. However our result does not require the full machinery of gaussian mixtures, and shows that the result only depends on the behaviour of the second derivative of the $h$'s at infinity. \hfill $\diamondsuit$
\end{remark}
\medskip

\begin{remark}\label{remD1}
In the previous proof we implicitly have used the fact that (H) is satisfied. We know that it is the case when $W \in C^2(\mathbb R^n)$. If $W$ is only continuous, we may replace $W$ by $W_\varepsilon=W*\gamma_\varepsilon$ where $\gamma_\varepsilon$ is a tiny centered gaussian density. $W_\varepsilon$ is still even and convex, so that the Theorem applies. Since the bound does not depend on $\varepsilon$ we may take limits in the corresponding Poincar\'e inequalities and get the same bound for $W$.
\hfill $\diamondsuit$
\end{remark}

\begin{remark}\label{remlargep}
Let now consider the case $p>2$. This time we have to control $$\mu\left(|x_i|\leq u^{-1/(p-2)}\right) \, ,$$ for large $u$'s.  

Using \eqref{eq1d} and since $\rho_i$ is even and log-concave, we see that $$\mu\left(|x_i|\leq u^{-1/(p-2)}\right) \leq Z_i^{-1} \, \rho_i(0) \, u^{-1/(p-2)} \, .$$ But $$ Z^{-1}_i \, \rho_i(0) \, = \,  \frac{\int \, e^{-W(x_1,...,x_{i-1},0,x_{i+1}, ...,x_n) -\sum_{j \neq i} |x_j|^p} \, \prod_{i\neq j} dx_j}{\int \, e^{-W(x) -\sum_{j} |x_j|^p}  \, \prod_{j} dx_j} \, .$$ Denote by 
\begin{equation}\label{eqmaxc}
\alpha = \max_i \, Z_i^{-1} \, \rho_i(0) \, .
\end{equation}
Then we get $$\beta(s)= \frac{1}{p(p-1)} \, \left(\frac{\alpha n}{s}\right)^{p-2}$$
 so that we have to estimate (choosing $\theta=1$ in \eqref{eqchoixs}) $$\int_2^{+\infty} \, \exp \left(- \frac{2 p(p-1)}{(\alpha n)^{p-2}} \, \frac{t}{\ln^{2(p-2)}(t)}\right) \, dt \, .$$ Using that $t/\ln^k(t)$ is bounded below by $c(k,\varepsilon) t^{1-\varepsilon}$ for any $\varepsilon >0$, we easily obtain
\begin{proposition}\label{thmBKp2}
Let $\mu(dx)=Z^{-1} \, e^{-W(x) -\sum_{i=1}^n |x_i|^p} \, dx$. We assume that $p>2$ and that $W$ is convex and even so that $\mu$ is log-concave. Then for all $\varepsilon >0$, there exists a constant $c(p,\varepsilon)$ such that $$C_P(\mu) \, \leq \, c(p,\varepsilon) \, (\alpha \, n)^{(p-2)(1+\varepsilon)}$$ where $$\alpha = \max_i \, \frac{\int \, e^{-W(x_1,...,x_{i-1},0,x_{i+1}, ...,x_n) -\sum_{j \neq i} |x_j|^p} \, \prod_{i\neq j} dx_j}{\int \, e^{-W(x) -\sum_{j} |x_j|^p}  \, \prod_{j} dx_j} \, .$$
\end{proposition}
For instance if we assume  that $t \mapsto W(x_1,...t,...,x_n)$ is $\beta$ H\"{o}lder continuous, uniformly in $x$ and $i$, using $|W(x_1,...t,...,x_n)-W(x_1,...0,...,x_n)|\leq L|t|^\beta$, we get $$\alpha \leq \frac{1}{\int \, e^{-L|t|^\beta - |t|^p} \, dt} \, .$$
\medskip

The previous result may have some interest only if $2<p<3$. This is also quite natural: for large $p$'s, $|x|^p$ becomes flat near the origin so that one cannot expect to use some convexity approach. 

The best general control (thus including the case $p>2$ for Subbotin distributions) is obtained in Theorem 18 of \cite{BK19}, and says that $$C_P(\mu) \leq c \, n \, \max_i(C_P(\nu_i)) \, .$$ In addition, in subsection 3.4 of \cite{BK19}, it is shown that the factor $n$ is optimal by considering log concave perturbations of Subbotin distributions $\nu_i$ with exponent $p$ for large $p$'s  for which $C_P(\mu)$ is at least of order $n^{(p-2)/p}$.  
\medskip

However if $W$ is unconditional (i.e. $W(\sigma x)=W(x)$ for all $\sigma \in \{-1,1\}^n$), one can deeply reinforce the previous result and show that $C_P(\mu) \leq \max_i(C_P(\nu_i))$ as shown in \cite{BK19} Theorem 17.
\hfill $\diamondsuit$
\end{remark}
\medskip

\begin{remark}\label{remcov}
Denote by $Cov(\mu)$ the covariance matrix, i.e. $Cov_{i,j}(\mu) = \mu(x_i x_j) - \mu(x_i)\mu(x_j)$. It is immediate that $\sigma^2(\mu)= ||Cov_\mu||^2_{HS} \leq C_P(\mu)$ ($\sigma(\mu)$ being the largest eigenvalue of $Cov(\mu)$). Our proof thus gives an universal bound (that does not depend on $W$) for the Covariance matrix. The proofs by Barthe and Klartag use first estimates for this covariance matrix. \hfill $\diamondsuit$
\end{remark}
\medskip

Looking at log-concave perturbations of log-concave product measures as above, can be partly motivated by statistical issues. We refer to \cite{CGperturb} (in particular the final section) for some of them. Of course looking at product measures is interesting thanks to the tensorization property of Poincar\'e inequality, furnishing dimension free bounds. For log-concave measures, another case is well understood since S. Bobkov's work \cite{bobsphere}, namely radial measures. The following version is due to M. Bonnefont, A.Joulin and Y. Ma (\cite{BJM} Theorem 1.2)
\begin{theorem}{\textit{(Bobkov, Bonnefont-Joulin-Ma)}}\label{thmbobsphere}

Let $\mu$ be a spherically symmetric (radial) log-concave probability measure on $\mathbb R^n$, $n \geq 2$. Then $$C_P(\mu) \leq \frac{\mu(|x|^2)}{n-1} \, .$$
\end{theorem}

We can obtain a result similar to proposition \ref{thmBK} or proposition \ref{thmBKp2}
\begin{theorem}\label{thmrad}
Let $\mu(dx) = Z_\mu^{-1} \, e^{-W(x) - h(|x|^2)} \, dx$ be a probability measure on $\mathbb R^n$. We assume that $W$ is even and convex and that $h$ is convex and non-decreasing on $\mathbb R^+$, so that $\mu$ is log-concave. $W$ and $h$ are also normalized so that $W(0)=h(0)=0$ (and consequently $W$ and $h$ are non-negative). Introduce $$\nu_h(dx) =  e^{-h(|x|^2)} \, dx \, .$$ There exists an universal constant $c$ such that $$C_P(\mu) \, \leq c \, \left(1 + \int_2^{+\infty} \, e^{-4 t \, h'\left(\frac{1}{(c_n(\mu) \, \ln^2(t))^{2/n}}\right)} \, dt\right) \, , $$ with $$c_n(\mu)= Z_\mu^{-1} \, \frac{\pi^{n/2}}{n \Gamma(n/2)} \leq \, \inf_\theta \left\{\frac{\pi^{n/2}}{n \Gamma(n/2)} \, \frac{e^{\max_{|x|=\theta}W(x)} }{\nu_h(|x|\leq \theta)}\right\} = \, \inf_\theta \left\{\frac{e^{\max_{|x|=\theta}W(x)} }{n \, \int_0^\theta r^{n-1} \, e^{-h(r^2)} \, dr}\right\} \, .$$
\end{theorem}
\begin{remark}\label{rempassimal}
Let $\mu_\lambda(dx)= Z_{\mu}^{-1} \, \lambda^{-n} \, e^{-W(x/\lambda) - h(|x|^2/\lambda^2)} \, dx$ a dilation of $\mu$. Notice that $\lambda^2 c_n^{2/n}(\mu_\lambda)=c_n(\mu)$. Since one has a factor $1/\lambda^2$ in front of $h'$, we partly recover the homogeneity of the Poincar\'e constant under dilations. \hfill $\diamondsuit$
\end{remark}

\begin{proof}
Once again we may assume that $W$ and $h$ are smooth, convolving with a tiny gaussian kernel, that preserves convexity and parity. For simplicity we also assume that $h'$ is (strictly) increasing, so that $h'$ is one to one.

For two vectors $x$ and $y$ we write $xy$ for the vector with coordinates $(xy)_i=x_i y_i$. It holds
\begin{eqnarray*}
\Gamma_2 f &=& \sum_{i,j} (\partial^2_{ij} f)^2 + \langle \nabla f, Hess W \nabla f\rangle + 4 \, h''(|x|^2) \, |x \nabla f|^2 + 2 \, h'(|x|^2) |\nabla f|^2 \\ &\geq& 2 \, h'(|x|^2) |\nabla f|^2
\end{eqnarray*}
so that 
\begin{equation}\label{eqrad1}
u \, \mu((Af)^2) + \mu\left(2 \, h'(|x|^2)\leq \frac 1u\right) \, |||\nabla f|^2||_\infty \, \geq \, \mu(|\nabla f|^2) \, .
\end{equation}
So $\mu$ satisfies a (WI $\Gamma_2$ grad) inequality, with $$\beta^{-1}(u) = \mu\left(2 \, h'(|x|^2)\leq \frac 1u\right) = \mu_r\left(h'(r^2) \leq \frac{1}{2u}\right)= \mu_r\left(r \leq \sqrt{(h')^{-1}(1/2u)}\right)$$ where $\mu_r$ denotes the probability distribution of the radial part of $\mu$. We have $$\mu_r(dv) = Z_\mu^{-1} \, n \, \omega_n \; v^{n-1} \, e^{-h(v)} \, \left(\int_{S^{n-1}} \, e^{-W(v \theta)} \, \sigma_n(d\theta)\right) dv$$ where $\sigma_n$ denotes the uniform measure on the sphere $S^{n-1}$ and $\omega_n= \frac{\pi^{n/2}}{n \Gamma(n/2)}$ denotes the volume of the unit (euclidean) ball. It follows, since $W$ and $h$ are non-negative, $$\mu_r\left(r \leq \sqrt{(h')^{-1}(1/2u)}\right) \leq Z_\mu^{-1} \, \frac{\pi^{n/2}}{n \, \Gamma(n/2)} \, ((h')^{-1}(1/2u))^{n/2}$$ from which we deduce that we can choose $$\beta(t) = \frac{1}{2 \, h'((s/c_n)^{2/n})} \quad \textrm{ with } \; c_n= Z_\mu^{-1} \, \frac{\pi^{n/2}}{n \Gamma(n/2)} \, .$$ It remains to apply proposition \ref{propwG2logconc}.
\medskip

The next step is thus to get some tractable bound for $c_n$, i.e a lower bound for $Z_\mu$. The simplest way to do it is to use the fact that $W$ is non-decreasing on each radial direction so that for all $\theta>0$  $$Z_\mu \geq \int_{|x|\leq \theta} \, e^{-W(x) -h(|x|^2)} \, dx \, \geq \, e^{- \max_{|x|=\theta}W(x)}  \,  \nu_h(|x|\leq \theta) \, .$$ 
\end{proof}

\begin{corollary}\label{corradp}
In particular if $h(u)= u^p$ with $p\geq 1$, we have $$C_P(\mu) \leq 12288 \, \ln(2) \,  \frac{c_n^{\frac{2(p-1)}{n}}(\mu)}{4 p} \, (4(p-1))^{\frac{4(p-1)}{n}} \, .$$ 
\end{corollary}
\begin{proof}
If $h(u)= u^p$ for $p>1$, the corresponding dilation $\mu_\lambda$ is given by $h_\lambda(u)=\lambda^{-2p} u^p$. Recall that $c_n(\mu_\lambda)=\lambda^{-n} c_n(\mu)$. 

 We shall use that $$\ln(t) \leq \frac{1}{\alpha \, 2^\alpha} \, t^\alpha \, +(\ln(2) -(1/\alpha)) \; \textrm{ for $t\geq 2$ and $\alpha >0$.}$$ If $t\geq 2$, we thus have $\ln(t) \leq \frac{1}{\alpha \, 2^\alpha} \, t^\alpha$ if $\alpha \leq 1$ and $\ln(t) \leq t^\alpha$ if $\alpha \geq 1$.
 
It follows $$\ln^{\frac{4(p-1)}{n}}(t) \leq c_\beta \, t^{\beta}$$ for $t\geq 2$ and $0<\beta$, with $$c_\beta = 2^{-\beta} \, \left(\frac{4(p-1)}{\beta n}\right)^{\frac{4(p-1)}{n}} \, \textrm{ if } \frac{\beta n}{4(p-1)} \leq 1 \quad ; \quad c_\beta=1 \textrm { if } \frac{\beta n}{4(p-1)} \geq 1 \, .$$ This yields $$e^{-4 t \, h_\lambda'\left(\frac{1}{(c_n(\mu_\lambda) \, \ln^2(t))^{2/n}}\right)} \leq e^{-\kappa_\beta \; t^{1-\beta}} $$ for $$\kappa_\beta= \frac{4  p}{c_\beta \, \lambda^2 \, c_n^{\frac{2(p-1)}{n}}(\mu)} \, .$$ A simple change of variables $u = \kappa_\beta \, t^{1-\beta}$, together with the positivity of all constants yields $$\int_2^{+\infty} \, e^{-4 t \, h_\lambda'\left(\frac{1}{(c_n(\mu_\lambda) \, \ln^2(t))^{2/n}}\right)} dt \leq ((1-\beta)\kappa_\beta)^{-1} \, \int_0^{+\infty} \, u^{\frac{\beta}{1-\beta}} \, e^{-u^{1-\beta}} \, du \, .$$ Choosing for simplicity $\beta=1/n$,  so that $\beta n/4(p-1) \leq 1$, for $n\geq 2$ the final integral is bounded independently of $n$ for instance by $c=\int_0^{+\infty} u \, e^{-\sqrt u} \, du=12$.  It follows $$C_P(\mu_\lambda) \leq 1024 \, \ln(2) \, \left(2 + c \, \frac{\lambda^2 \, c_n^{\frac{2(p-1)}{n}}(\mu)}{4 p} \, (4(p-1))^{\frac{4(p-1)}{n}}\right) \, .$$ Using $C_P(\mu)=\lambda^{-2} \, C_P(\mu_\lambda)$ and letting $\lambda$ go to infinity furnishes the result.
 
 For $p=1$ the result follows from strict convexity.
\end{proof}
\medskip

\begin{remark}\label{remradial}
If $\mu_r$ denotes the radial distribution of $\mu$, $$\mu_r(dv) = \rho(v) \, v^{n-1} \, e^{-h(v)} \, .$$ $\rho$ is clearly an even function. Since for a fixed $\theta$, $v \mapsto W(v \theta)$ is even and convex, it is non-decreasing, so that $v \mapsto \rho(v)$ is non-increasing. $\rho$ is non necessarily log-concave, but we can again apply Proposition 6 in \cite{BK19} furnishing, with $\bar \nu_h=Z^{-1} \, \nu_h$,
\begin{equation}\label{eqrad2}
C_P(\mu_r) \, \leq \, C_P(\bar\nu_h) \, .
\end{equation}
The measure $\bar\nu_h$ being log-concave, we know that $$C_P(\bar\nu_h) \leq 12 \, \Var_{\bar\nu_h}(v) \, .$$ What is important here is that the Poincar\'e constant of the radial measure $\mu_r$ can be bounded independently of $W$.
\hfill $\diamondsuit$
\end{remark}

\begin{remark}\label{remgoodorder}
Is the bound in Corollary \ref{corradp} of the good order ? To see it look at the particular case $W=0$. In this case $Z_\mu= \frac{n \, \omega_n}{2p} \, \Gamma(n/2p)$ so that $c_n(\mu)=\frac{2p}{n \, \Gamma(n/2p)} \, ,$ and our bound furnishes $$C_P(\mu) \leq c \, \frac{2^{10(p-1)/n} \, p^{6(p-1)/n}}{4p \, n^{2(p-1)/n}}  \, \frac{1}{\Gamma^{2(p-1)/n}(n/2p)} \, ,$$ for some universal $c$. In this case the following very precise bounds were obtained by Bonnefont, Joulin and Ma in \cite{BJM}, $$\frac{\mu(|x|^2)}{n} \leq C_P(\mu) \leq \frac{\mu(|x|^2)}{n-1} \, .$$ Since $$\mu(|x|^2)= \frac{\Gamma((n+2)/2p)}{\Gamma(n/2p)}$$ for $n/p \gg 1$ (even for large $p$'s) we may use $$\Gamma(z) \sim_{z \to +\infty} \sqrt{2\pi} z^{z-(1/2)} \, e^{-z}$$ so that the Bonnefont, Joulin, Ma theorem furnishes 
\begin{equation}\label{eqBJMrad}
C_P(\mu) \sim (2ep)^{-1/p} \, n^{1-(1/p)} \, .
\end{equation}
For the same asymptotics our bound furnishes (for some new constant $c$)
\begin{equation}\label{eqBJMradrev}
C_P(\mu) \leq c \, p^{3(p-1)/n} \, n^{1-(1/p)} \, .
\end{equation}
Hence provided $p \ln(p) \leq C n$, we get the good order (but of course not the good constant). This shows that our bound is not so bad. \hfill $\diamondsuit$
\end{remark}
\medskip

\section{The case of compactly supported measures.}\label{seccomp}

Let us come back to the proof of Proposition \ref{propledspin} starting with 
\begin{equation}\label{eqG2perturb}
\Gamma_2 f = \sum_{i,j} (\partial^2_{ij} f)^2 + \sum_i h''_i(x_i) (\partial_i f)^2 + \langle \nabla f , Hess W \, \nabla f\rangle \, .
\end{equation} 
If $W$ is convex, we thus have
\begin{eqnarray}\label{eqG2perturb2}
\mu(\Gamma_2 f) &\geq& \sum_{i} \mu((\partial^2_{ii} f)^2 + \sum_i h''_i(x_i) (\partial_i f)^2 ) \, \nonumber \\ &=& \sum_{i} \mu(\eta_{i,x}((\partial^2_{ii} f)^2 +  h''_i(x_i) (\partial_i f)^2 ))
\end{eqnarray} 
Instead of adding and substracting $\partial^2_{ii} W  \, (\partial_i f)^2$, consider $\eta_{i,x}$ as a perturbation of $$\theta_i(dt)=z_i^{-1} \, e^{-h_i(t)} dt$$ using the notation $$\eta_{i,x}(dt) = Z_{i,x}^{-1} \, e^{- W_{i,x}(t)} \, \theta_i(dt) \, .$$ Since we integrate a non-negative quantity it holds 
\begin{eqnarray}\label{eqG2perturb3}
\eta_{i,x}((\partial^2_{ii} f)^2 +  h''_i(x_i) (\partial_i f)^2 ) &\geq& e^{- \sup W_{i,x}} \, \theta_i((\partial^2_{ii} f)^2 +  h''_i(x_i) (\partial_i f)^2 ) \nonumber \\ &\geq& e^{- \sup W_{i,x}} \, SG(\theta_i) \; \theta_i((\partial_i f)^2 ) \nonumber \\ &\geq& e^{- \Osc W_{i,x}} \, SG(\theta_i) \; \eta_{i,x}((\partial_i f)^2 ) \, ,
\end{eqnarray}
provided $$\theta_i(\Gamma_2 g)=\theta_i((L_i g)^2)$$ with $L_i g =g'' - h_i' g'$. Notice since $W_{i,x}$ is convex, its Oscillation cannot be bounded on $\mathbb R$, unless $W_{i,x}$ is constant. Hence the previous result has no interest on $\mathbb R^n$ and we shall only consider the case where the process lives in a bounded domain $D$.

We have thus obtained some variation of the renowned Holley-Stroock perturbation result namely
\begin{proposition}\label{propHSmarg}
Let $\mu(dx)=Z^{-1} \, e^{-W(x) -\sum_{i=1}^n \, h_i(x_i)} \, \mathbf 1_D(x) \, dx$ be a probability measure on the hypercube $D=\prod_i ]a_i,b_i[$ . Assume that 
\begin{enumerate}
\item[(G1)] For all $i$, the one dimensional diffusion $dy_t^i = \sqrt 2 \, dB_t^i - h_i'(y_t^i) dt$ satisfies (H) on $]a_i,b_i[$ with reversible measure $\theta_i(du)=z_i^{-1} \, e^{-h_i(u)} \, \mathbf 1_{u \in]a_i,b_i[} \, du$.
\item [(G2)] $W \in C^\infty(\mathbb R^n)$ and is convex.
\end{enumerate}
Introduce the one dimensional conditional log-density  $$W_{i,x}(t) = W(x_1,...,x_{i-1},t,x_i,..x_n) \, .$$ Then $$C_P(\mu) \leq \max _i \, \sup_{x} \, e^{\Osc (W_{i,x})} \; \max_i \, C_P(\theta_i) \, .$$
\end{proposition}
Since $\max _i \, \sup_{x} \, e^{\Osc (W_{i,x})} \leq \Osc W$ we recover (provided $W$ is convex) Holley-Stroock result for a product reference measure on a hypercube. But what is important here is that we only have to consider the Oscillation of $W$ along lines parallel to the axes. 
\begin{proof}
The only thing remaining to prove is that we can work with $C_b^\infty(D)$ functions $f$ so that $u \mapsto f(x_1,...,x_{i-1},u,x_{i+1}, ...,x_n)$ is also $C_b^\infty(]a_i,b_i[)$ and we may use (G1) to justify the calculations we have done before. It is thus enough to show that (H) is satisfied for the full process i.e. with $V=W + \sum_i h_i$. 

Since $\nabla W$ and $\Delta W$ are bounded on $\bar D$, the law of $X_.^x$ is absolutely continuous w.r.t. to the one of $(y^1_.,...,y^n_.)$ thanks to Girsanov theory. It follows that the exit time of $D$ is almost surely infinite since the same holds for $(y^1_.,...,y^n_.)$ according to (G1). In addition the Feynman-Kac representation of the density $F_T$ (on $C^0([0,T],D)$) is again given by the formula of Example \ref{example1}, so that, as we have seen, (H) is satisfied. 
\end{proof}
\smallskip

\begin{corollary}\label{corcomp}
Let $\mu(dx)=Z^{-1} \, e^{-W(x) -\sum_{i=1}^n \, h_i(x_i)} \, \mathbf 1_D(x) \, dx$ be a probability measure on the hypercube $D=\prod_i ]a_i,b_i[$ . Assume that the $h_i$'s and $W$ are convex and $C_b^2(\bar D)$. Then, with the notations of Proposition \ref{propHSmarg} we have $$C_P(\mu) \leq 12 \, \max _i \, \sup_{x} \, e^{\Osc (W_{i,x})} \; \max_i \, C_P(\theta_i) \, .$$
\end{corollary}
\begin{proof}
As usual, using smooth approximations, we may assume that $W \in C^\infty(\mathbb R^n)$. We shall perturb $\mu$ in order to apply the previous proposition. To this end, on the interval $]a_i,b_i[$ define $$h^i_\varepsilon(u)  = \varepsilon \, \left(\frac{1}{u-a_i} + \frac{1}{b_i-u}\right) \, .$$  Consider $$\mu_\varepsilon(dx) = Z^{-1} \, e^{-W(x) \, - \sum_i \, (h_i(x_i) + h^i_\varepsilon(x_i))} \, \mathbf 1_D(x) \, dx \, .$$ Denote $g^i_\varepsilon=h_i+h^i_\varepsilon$. 

Assumptions (G1) and (G2) of proposition \ref{propHSmarg} are satisfied. We already assumed (G2). In order to show (G1) it is first enough to use Feller test of non explosion for a one dimensional diffusion, i.e. to check that for $c_i=\frac 12 (a_i+b_i)$, $$\int_{c_i}^{a_i} \, \exp \left(\int_{c_i}^y \, (g^i_\varepsilon)'(u) du\right) dy = - \infty $$ (replacing $a_i$ by $b_i$ we similarly get $+\infty$)  according for instance to \cite{IW} Chapter VI, Theorem 3.1, which is immediate. It follows that \eqref{eqsortie} is satisfied. In addition $(g^i_\varepsilon)' \in \mathbb L^2(\theta^i_\varepsilon(du))$ where $\theta^i_\varepsilon(du) = z_\varepsilon^{-1} \, e^{-g^i_\varepsilon(u)} \, \mathbf 1_{]a_i,b_i[}(u) \, du$, so that we are in the situation of Example \ref{example2} ensuring that the one dimensional $y_.$ in (G1) satisfies (H).

We have thus obtained $$C_P(\mu_\varepsilon) \leq \max _i \, \sup_{x} \, e^{\Osc (W_{i,x})} \; \max_i \, C_P(\theta^i_\varepsilon(dt)) \, .$$ Using Lebesgue's bounded convergence Theorem, for all $f \in C_b^0(D)$ it holds $$\lim_{\varepsilon \to 0} \, \int_D \, f(x) \, e^{-W(x) -\sum_{i=1}^n \, g^i_\varepsilon(x_i)} \, dx \, = \, \int_D \, f(x) \, e^{-W(x) -\sum_{i=1}^n \, h_i(x_i)} \, dx$$ so that using this result for $f$ and $1$, $\mu_\varepsilon$ weakly converges to $\mu$. It follows $$C_P(\mu) \leq \liminf_{\varepsilon \to 0} C_P(\mu_\varepsilon) \leq \max _i \, \sup_{x} \, e^{\Osc (W_{i,x})} \; \liminf_{\varepsilon \to 0}\max_i \, C_P(\theta^i_\varepsilon) \, .$$ We may now use the fact that $\theta^i_\varepsilon$ is log-concave since both $h_i$ and $h^i_\varepsilon$ are convex. We thus have $$C_P(\theta^i_\varepsilon) \leq 12 \, \Var_{\theta^i_\varepsilon}(x_i) \, .$$ Once again $\theta^i_\varepsilon$ weakly converges to $\theta_i$ and since $x_i \mapsto x_i^2$ is continuous and bounded on $[a_i,b_i]$, $$\Var_{\theta^i_\varepsilon}(x_i) \to \Var_{\theta_i}(x_i)$$ so that the conclusion follows from the immediate $Var_{\theta_i}(x_i) \leq C_P(\theta_i)$.
\end{proof}
\medskip

\section{Super $\Gamma_2$ condition}

As there are weak Poincar\'e inequalities, Super Poincar\'e inequalities (SPI) have also been  introduced by Wang \cite{W00} as a concise description of functional inequalities strictly stronger than Poincar\'e inequalities, in particular logarithmic Sobolev (or more generally $F$-Sobolev) inequalities.
\medskip

(SPI) is often written in  the following form: $\forall s>0$, there exists a non-increasing $\beta: ]0,\infty[ \mapsto [1,+\infty[$  such that
\begin{equation}
\label{SPI}
\mu(f^2)\le s\mu(|\nabla f|^2)+\beta(s) \mu(|f|)^2.
\end{equation}
Applying \eqref{SPI} to constant functions one sees that $\beta(s)\geq 1$ for all $s$. Since $1$ is assumed to belong to the range of $\beta$, the (SPI) inequality implies a Poincar\'e inequality with $C_P(\mu) \leq \beta^{-1}(1)$, and one has $\beta(s)=1$ for $s\ge C_P(\mu)$ . When $\beta(s)=ae^{b/s}$ for positive $a$ and $b$, then the Super Poincar\'e inequality is equivalent to a logarithmic Sobolev inequality (see \cite{CGWrad} lemma 2.5 and lemma 2.6 for a precise statement). \\
It is also possible to consider SPI with a $L^p$ norm rather than the $L^1$ norm, so that we will introduce general (p-SPI) for $1\le p<2$ and all $s>0$,
\begin{equation}
\label{gSPI}
\mu(f^2)\le s\mu(|\nabla f|^2)+\beta(s) \mu(|f|^p)^{2/p}.
\end{equation}
This time, \eqref{gSPI} does not imply a Poincar\'e inequality, so that it is natural to assume in addition that $C_P(\mu) \leq +\infty$. In this case we have the following
\begin{lemma}\label{lemspicentre}
Assume that $C_P(\mu)<+\infty$ and that the following centered (cp-SPI) inequality is satisfied for all $s>0$,
\begin{equation}\label{cgSPI}
\Var_\mu(f)\le s\mu(|\nabla f|^2)+\beta_c(s) \mu(|f-\mu(f)|^p)^{2/p} \, ,
\end{equation}
where $\beta$ is non increasing. Then (p-SPI) holds with $\beta(s)=1+4\beta_c(s)$.
\end{lemma}
\begin{proof}
Since $C_P(\mu)<+\infty$ we may choose $\beta_c(s)=0$ for $s>C_P(\mu)$. Let $f$ be given. It holds
$$\mu(|f-\mu(f)|^p) \leq 2^{p-1} \, (\mu(|f|^p)+\mu^p(|f|)) \leq \, 2^p \, \mu(|f|^p)$$ yielding 
\begin{eqnarray*}
\mu(|f|^2) &=& \Var_\mu(f) + \mu^2(f) \leq  s\mu(|\nabla f|^2) + \beta_c(s)  \mu^{2/p}(|f-\mu(f)|^p) + \mu^2(|f|) \\ &\leq&  s\mu(|\nabla f|^2) + (4 \beta_c(s)+1) \, \mu^{2/p}(|f|^p) \, .
\end{eqnarray*}
\end{proof}

It is then natural to introduce an integrated super $\Gamma_2$ condition: for some $1\le p<2$, there exists a positive non-increasing function $\beta$ such that $\forall s>0$ 
$$(pSI-\Gamma_2)\qquad \mu(|\nabla f|^2)\le s\,\mu(Af)^2)+\beta(s)  \mu(|f|^p)^{2/p}.$$
In the sequel we assume that $C_P(\mu)<+\infty$, so that for all $s\ge C_P(\mu)$ one may take $\beta(s)=0$. 

Let us begin by this simple proposition

\begin{proposition}\label{propsgietc}
We have the following
\begin{enumerate}
\item A $(p-SPI)$ inequality is equivalent to
\begin{equation}
\label{SPIcv}
\mu((P_tf)^2\le e^{-2t/s}\mu(f^2)+\beta(s)\mu(|f|^p)^{2/p}(1-e^{-2t/s}),
\end{equation}
for all $s>0$ and all $t\geq 0$.
\item A $(pSI-\Gamma_2)$ condition is equivalent to 
\begin{equation}
\label{Sgam-cv}
\mu(|\nabla P_tf|^2\le e^{-2t/s}\mu(|\nabla f|^2)+\beta(s)\mu(|f|^p)^{2/p}(1-e^{-2t/s}).
\end{equation}
for all $s>0$ and all $t\geq 0$.
\end{enumerate}
\end{proposition}
\begin{proof}
The first part is well known and is included in Wang's work \cite{W00}. The second point will follow the same line of proof. As already emphasized in the previous sections, denoting
$$F(t)=\mu(|\nabla P_t f|^2)$$
one has $$F'(t)=-2\mu((A P_tf)^2)$$
so that the $(pSI-\Gamma_2)$ gives directly
$$F'(t)\le -\frac 2s F(t)+\frac{2\beta(s)}{s} \mu(|P_tf|^p)^{2/p}$$
and since $\mu(|P_tf|^p)^{2/p}\le \mu(|f|^p)^{2/p}$ we conclude  thanks to Gronwall's lemma. The other implication comes from differentiating with respect to time at time 0.
\end{proof}

\subsection{From (p-SPI) to (pSI-$\Gamma_2$). \\ }

We follow the same proof as in section \ref{secweak}, assuming that a $(p-SPI)$ holds, i.e. we use Cauchy-Schwartz inequality in order to get
\begin{eqnarray*}
\mu(|\nabla f|^2)&=&\mu(-fAf)\\
&\le&\sqrt{\mu(f^2)\,\mu((Af)^2)}\\
&\le&\left(s\,\mu (|\nabla f|^2)\,\mu((Af)^2)+\beta(s)\mu(|f|^p)^{2/p}\mu((Af)^2)\right)^{1/2}.
\end{eqnarray*}
Recall now the already used following fact:  if $0\le u\le \sqrt{Au+B}$ then $u\le A+B^{1/2}$.  It yields
\begin{eqnarray*}
\mu(|\nabla f|^2)&\le&s\,\mu((Af)^2)+\sqrt{\beta(s)\mu(|f|^p)^{2/p}\mu((Af)^2)}\\
&\le&\frac32 s\,\mu((Af)^2)+\frac{\beta(s)}{2s}\mu(|f|^p)^{2/p}.
\end{eqnarray*}
We thus see that we have ``lost'' a factor $1/s$ but if we think to the logarithmic Sobolev inequality, it roughly means the loss of a constant.

\subsection{From (pSI-$\Gamma_2$) to (p-SPI).\\}

Starting with $$\mu(|\nabla P_tf|^2\le e^{-2t/s}\mu(|\nabla f|^2)+\beta(s)\mu(|P_tf|^p)^{2/p}(1-e^{-2t/s})$$
and using
$$\Var_\mu(f)=2 \int_0^\infty \mu(|\nabla P_uf|^2)du$$
we get
$$\Var_\mu(f)\le 2 \int_0^\infty e^{-2u/s}\mu(|\nabla f|^2)du+2\beta(s)\int_0^\infty \mu(|P_uf|^p)^{2/p}(1-e^{-2u/s})du.$$
Assume first that $f$ is centered. If $p>1$ then Poincar\'e inequality implies back an exponential convergence in $L^p$ norm (see \cite{CGR} Theorem 1.3) so that for all centered $f$ we get
$$\mu(f^2)\le s\mu(|\nabla f|^2)+K_p\beta(s)\mu(|f|^p)^{2/p}$$
where $K_p$ depends on $p$ and is going to infinity as $p$ goes to $1$. Applying lemma \ref{lemspicentre} we thus obtain $$\mu(|f|^2) \leq s\mu(|\nabla f|^2)+(1+4 K_p \beta(s)) \mu(|f|^p)^{2/p} \, .$$
\medskip

\section{Appendix: about the heart of darkness.}\label{secdark}

Let us come back to the framework of this work especially Proposition \ref{propESA}. 

First of all, if (H) is satisfied, according to Theorem 2.2.25 and its proof in Royer's book \cite{RoyF} (also see the english version \cite{RoyE}), the following  holds
\begin{enumerate}
\item[(A1)]  $P_t$ extends to a $\mu$-symmetric continuous  Markov semi-group $e^{-t \tilde A}$ on $\mathbb L^2(\mu)$. We denote by $\mathcal D(\tilde A)$ the domain of the generator $\tilde A$ of this $\mathbb L^2(\mu)$ semi-group.
\item[(A2)]  Any $f \in C^2(D)$ such that $|\nabla f|$ is bounded and $Af \in \mathbb L^2(\mu)$ belongs to $\mathcal D(\tilde A)$, and $\tilde A f=Af$. 
\item[(A3)]  If $f \in \mathcal D(\tilde A)$ the set of Schwartz distributions on $D$, then  $f \in \mathcal D'(D)$ and satisfies $\tilde A f=Af$ in $\mathcal D'(D)$.
\end{enumerate} 
Actually Royer only considers the case $D=\mathbb R^n$, but the key point in the proof is that one can apply Ito's formula for such an $f$ up to time $t$ (without any stopping time) which is ensured by the conservativeness in $D$. 
\medskip

In the case $D=\mathbb R^n$ the proof of ESA for $C_0^\infty$ the set of smooth compactly supported functions is contained in \cite{Wiel} using an elliptic regularity result Theorem 2.1 in \cite{Fre} (actually the latest result certainly appeared in other places). The proof is explained in Theorem 2.2.7 of \cite{RoyF} (also see Proposition 3.2.1 in \cite{BaGLbook}) when $V$ is $C^\infty$. The structure of $\mathcal D(\tilde A)$ is also proved in the same Theorem. 

We shall explain the proof when $D$ is bounded, still assuming for simplicity that $V \in C^\infty(D)$. The same elliptic regularity should be used to extend the result to $V \in C^2(D)$, but will introduce too much intricacies to be explained here. 
\begin{proof}
First consider the Dirichlet form $\mathcal E(f,g)=\mu(\langle \nabla f,\nabla g\rangle)$ whose domain is the closure of $C^\infty_0(D)$ denoted by $H^1_0(\mu,D)$. Since $\mathcal E$ is regular, Fukushima's theory (see \cite{Fuku}) allows us to build a symmetric Hunt process associated to $(\mathcal E,H^1_0(\mu,D))$. This process is then a solution to the martingale problem associated to $A$ and $C^\infty_0(D)$. Since $T_\partial^x$ is almost surely infinite, this martingale problem has an unique solution given by the (distribution) of the stochastic process $X_.^x$. 

In order to prove ESA it is enough to show that if $g \in \mathbb L^2(\mu)$ satisfies $\mu(g \, (A\varphi - \varphi))=0$ for all $\varphi \in C_0^\infty(D)$ then $g$ vanishes (see the beginning of the proof in \cite{RoyF} p.31). According to the proof in \cite{RoyF} p.31, it implies in particular that $g \in \mathcal D'(D)$ and satisfies $Ag=g$. Using that $A$ is hypoelliptic since $V\in C^\infty(D)$, we deduce that $g \in C^\infty(D)$. 

Using Ito's formula (since the process is conservative) we have 
\begin{equation}\label{eqmart}
\sqrt 2 \, \int_0^t \, \langle \nabla g(X_s),dB_s\rangle = g(X_t) - g(X_0) - \int_0^t \, g(X_s) \, ds
\end{equation}
almost surely. If $X_0$ is distributed according to $\mu$, the right hand side belongs to $\mathbb L^2(\mathbb P)$ ($\mathbb P$ being the underlying probability measure on the path space), so that the left hand side also belongs to $\mathbb L^2(\mathbb P)$. The $\mathbb L^2$ norm of this left hand side is equal to $2t \, \mu(|\nabla g|^2)$ so that $\nabla g \in \mathbb L^2(\mu)$.  

As a consequence $$t \mapsto \int_0^t \, \langle \nabla g(X_s),dB_s\rangle$$ is a $\mathbb P$ martingale so that for all bounded $h$, $$\mathbb E(h(X_0) g(X_t))= \mu(g h) +\int_0^t \mathbb E (h(X_0) g(X_s)) \, ds \, .$$ Since a regular disintegration of $\mathbb P$ is furnished by the distribution of the $X_.^x$'s, it follows $$P_tg = g + \int_0^t P_sg \, ds$$ $\mu$ almost surely, so that $g \in \mathcal D(\tilde A)$ and satisfies $\tilde Ag=g$. Hence $$\mu(g^2) = \mu(g Ag) = - \mu(|\nabla g|^2)$$ so that $g=0$.

The proof of the remaining part of the Theorem is the same as in \cite{RoyF} p.42.
\end{proof}

Finally we will indicate how to show that the semi-group is ergodic when $D$ is bounded (we already mention a possible way for $D=\mathbb R^n$ in section \ref{secweak}) . If $P_tf=f$ for all $t>0$ it follows that $f \in \mathcal D(A)$ and satisfies $Af=0$ so that $f$ is smooth thanks to hypoellipticity. Applying Ito's formula we deduce that $f(X_t^x)=f(x)$ a.s. for all $t>0$. Thanks to the Support Theorem (\cite{IW} chapter 6 section 8) we know that the distribution of $X_t^x$ admits a positive density w.r.t. Lebesgue measure, so that if $f(y)\neq f(x)$ for some $y$, hence all $z$ in a neighborhood $N$ of $y$ by continuity, $\mathbb P(X_t^x \in N)>0$ and thus $f(X_t^x) \neq f(x)$ with positive probability, which is a contradiction. 
\bigskip

Let us give now some of the most important examples. In these examples we assume that $V \in C^3$.
\medskip

\begin{example}\label{example1}
 If either 
\begin{enumerate}
\item[(H1)] \quad $V(x) \to +\infty$ as $x \to \partial D$ (i.e. $|x| \to +\infty$ if $D=\mathbb R^n$), and $\frac 12 \, |\nabla V|^2 - \Delta V$ is bounded from below, or
\item [(H2)] $D=\mathbb R^n$ and \quad $\langle x,\nabla V(x)\rangle \geq - a|x|^2 -b$ for some $a,b$ in $\mathbb R$,
\end{enumerate}
then  (H) is satisfied. If $V$ is convex (H2) is satisfied with $a=b=0$.

\noindent  If $D=\mathbb R^n$ these two cases are  detailed in \cite{RoyF} subsection 2.2.2 (conservativeness is shown in Theorem 2.2.19 therein). In the (H1) case for a bounded domain the only thing to do is to replace the exit times of large balls by the $T_k$'s in Lemme 2.2.21 of \cite{RoyF}. 
\medskip

In all cases the law of $(X^x_t)_{t \leq T}$ is given by $dQ = F_T dP$ where $P$ is the law of a Wiener process starting from $0$ and 
\begin{equation}\label{eqgirsanov}
F_T = \exp \left(\frac 12 V(x) - \frac 12 V(x +\sqrt 2 \, W_{T}) + \frac 12 \, \int_0^{T} \, (\frac 12 \, |\nabla V|^2 - \Delta V)(x + \sqrt 2 \, W_s) ds\right) \, .
\end{equation}
\hfill $\diamondsuit$
\end{example}
\medskip

\begin{example}\label{example2}

Assume now 
\begin{equation}\label{eqentropy}
\mu(|\nabla V|^2) < + \infty \, .
\end{equation}

\eqref{eqentropy} is an entropy condition related to the stationary Nelson processes (see \cite{MZ1,MZ2,CL1,CL1c,CL3}). The stationary (symmetric) conservative diffusion process is built in these papers. Conservative means here that 
\begin{equation}\label{eqsortieentrop}
T_{\partial D}=+\infty \quad \mathbb P_\mu \; \textrm {a.s}
\end{equation}
i.e. if $X_0$ is distributed according to $\mu$.

The proof  in the bounded case is a simple modification of the one in \cite{CL1}. The modification is as follows (we refer to the notations therein): 
\begin{enumerate}
\item First the flow $\nu_t$ is stationary with $\nu_t=\mu$.
\item Next the drift $B=-\nabla V$. Assuming in addition that $D$ has a smooth boundary, one can approximate in $\mathbb L^2(\mu)$, $B$ by $B_k$'s which are $C_b^\infty(\mathbb R^n)$ and coincide with $B$ on $\bar D_k=\{x \, ; , d(x,\partial D)\geq 1/k\}$ (for this we need $\partial D_k$ to be smooth). 
\end{enumerate}
One can then follow the ``Outline of proof'' (4.9 bis) in \cite{CL1} replacing the $T_n$ therein by the $T_k$ we have introduced before (the exit times of $D_k$) so that (4.14) in \cite{CL1} is trivially satisfied. (4.16) is then justified when \eqref{eqsortie} is satisfied. The only remaining thing to prove is thus (4.10) in \cite{CL1}. For $f \in C_0^\infty(\mathbb R^n)$ whose support contains $\bar D$ we may then proceed as in the proof of Theorem (4.18) in \cite{CL1} in order to prove it. 
\medskip

In order to show the strong existence of the diffusion process starting from $x$ (and not the stationary measure) it is enough to show \eqref{eqsortie} is satisfied (the strong existence of the diffusion process up to $T^x_{\partial D}$ is ensured since $V$ is local-Lipschitz). Since the stationary process is conservative, so is $X_.^x$ for $\mu$, hence Lebesgue,  almost all $x$. Standard results in Dirichlet forms theory show that this result extends to all $x$ outside some polar set. Actually it is true for all $x$ using the following (itself more or less standard 40 years ago): choose a small ball $B(x,\varepsilon)$ with $\varepsilon < d(x,\partial D)/2$ and introduce $S$ the exit time of this ball. For $t>0$ the distribution of $X_t \, \mathbf 1_{t<S}$ has a density w.r.t. Lebesgue's measure restricted to the ball (using e.g. Malliavin calculus). It follows from the Markov property and \eqref{eqsortieentrop} that $\mathbb P_x(T_{\partial D}^x < +\infty \; , \; t<S) =0$. Hence $$\mathbb P_x(T_{\partial D}^x = +\infty ) \leq \mathbb P_x(t<S)$$ for all $t>0$, the latter going to $0$ as $t \to 0$.
\medskip

In all cases the Feynman-Kac representation of $F_T$ in \eqref{eqgirsanov} is obtained by using Ito's formula with $V$  which is allowed since $V \in C^3(D)$ and \eqref{eqsortie} again.
\hfill $\diamondsuit$
\end{example}
\medskip

\begin{example}\label{example3}
If we do no more assume that the hitting time of the boundary is infinite, assumption (H) is not satisfied. The space of interest should be $H_b^1(\mu)$ the closure of $C_b^\infty(D)$ for the Dirichlet form. The corresponding process is the symmetric reflected diffusion process. A good reference is \cite{PW} where this normally reflected diffusion process is built (under much more general conditions). Assume that the boundary is smooth.

A little bit more is needed. First if $f \in H_b^1(\mu)$ and  $g \in \mathcal D(A)$, one has, according to Fukushima's theory (see \cite{Fuku} (1.3.10)), 
\begin{equation}\label{eqfuku}
\mu(\langle \nabla f,\nabla g\rangle) = - \, \mu(f \, Ag) \,  .
\end{equation}
If $f$ is smooth ($C^2(\bar D)$) and belongs to $\mathcal D(A)$ it also holds $$- \, \mu(f \, Ag) = \mu(\langle \nabla f,\nabla g\rangle) + \int_{\partial D} \, g \, \langle n_D,\nabla f\rangle \, e^{-V} \, d\sigma_D$$ according to  Green's identity. Here $n_D$ denotes the normalized inward normal vector on $\partial D$ and $\sigma_D$ denotes the surface measure on $\partial D$. Since the set of the traces on $\partial D$ of bounded functions in $H^1(\mu)$ is dense in $\mathbb L^\infty(\sigma_D)$, we deduce that $$\langle n_D,\nabla f\rangle_{|\partial D} = 0 \, .$$ 
It is however not clear in general that $P_t f$ is smooth (even if $V$ is). If one assumes that $\partial D$ is $C^\infty$ and $V \in C^\infty(\bar D)$, $P_tf \in C^\infty(\bar D)$ is shown in \cite{Catfourier} Theorem 2.9 by using the method of \cite{Catbord} (see the proof of Theorem 2.11 therein). Notice that the proof of regularity is using Sobolev imbedding theorem, so that one should relax the $C^\infty$ assumption but with dimension dependent regularity assumptions. Other more important difficulties will be pointed out later. 

The other major difficulty is that ESA is not satisfied in general. 
\medskip

In comparison with the previous example, the boundary term will disappear if $e^{-V}=0$ on $\partial D$. It is what happens if (H) is satisfied, but here again we do not need such a proof twhich is not useful in the present work.
\hfill $\diamondsuit$
\end{example}
\medskip

\bigskip

\bibliographystyle{plain}
\bibliography{G2bis}

\begin{thebibliography}{10}

\bibitem{Alonbast}
D.~Alonso-Gutierrez and J.~Bastero.
\newblock {\em Approaching the {K}annan-{L}ovasz-{S}imonovits and variance
  conjectures}, volume 2131 of {\em LNM}.
\newblock Springer, 2015.

\bibitem{BBCG}
D.~Bakry, F.~Barthe, P.~Cattiaux, and A.~Guillin.
\newblock A simple proof of the {P}oincar\'e inequality for a large class of
  probability measures.
\newblock {\em Elec. {C}omm. in {P}rob.}, 13:60--66, 2008.

\bibitem{BE}
D.~Bakry and M.~\'{E}mery.
\newblock Diffusions hypercontractives.
\newblock In {\em S\'{e}minaire de probabilit\'{e}s, {XIX}, 1983/84}, volume
  1123 of {\em Lecture Notes in Math.}, pages 177--206. Springer, Berlin, 1985.

\bibitem{BaGLbook}
D.~Bakry, I.~Gentil, and M.~Ledoux.
\newblock {\em Analysis and {G}eometry of {M}arkov diffusion operators.},
  volume 348 of {\em Grundlehren der mathematischen {W}issenchaften}.
\newblock Springer, Berlin, 2014.

\bibitem{BCR2}
F.~Barthe, P.~Cattiaux, and C.~Roberto.
\newblock Concentration for independent random variables with heavy tails.
\newblock {\em AMRX}, 2005(2):39--60, 2005.

\bibitem{BCO}
F.~Barthe and D.~Cordero-Erausquin.
\newblock Invariances in variance estimates.
\newblock {\em Proc. Lond. Math. Soc. (3)}, 106(1):33--64, 2013.

\bibitem{BK19}
F.~Barthe and B.~Klartag.
\newblock Spectral gaps, symmetries and log-concave perturbations.
\newblock {\em Bull. Hellenic Math. Soc.}, 64:1--31, 2020.

\bibitem{bobled}
S.~Bobkov and M.~Ledoux.
\newblock Poincar\'{e}'s inequalities and {T}alagrand's concentration
  phenomenon for the exponential distribution.
\newblock {\em Probab. Theory Related Fields}, 107(3):383--400, 1997.

\bibitem{bob99}
S.~G. Bobkov.
\newblock Isoperimetric and analytic inequalities for log-concave probability
  measures.
\newblock {\em Ann. Probab.}, 27(4):1903--1921, 1999.

\bibitem{bobsphere}
S.~G. Bobkov.
\newblock Spectral gap and concentration for some spherically symmetric
  probability measures.
\newblock In {\em Geometric aspects of functional analysis, Israel Seminar
  2000-2001,}, volume 1807 of {\em Lecture Notes in Math.}, pages 37--43.
  Springer, Berlin, 2003.

\bibitem{BJMsubbot}
M.~Bonnefont, A.~Joulin, and Y.~Ma.
\newblock A note on spectral gap and weighted {P}oincar\'{e} inequalities for
  some one-dimensional diffusions.
\newblock {\em ESAIM Probab. Stat.}, 20:18--29, 2016.

\bibitem{BJM}
M.~Bonnefont, A.~Joulin, and Y.~Ma.
\newblock Spectral gap for spherically symmetric log-concave probability
  measures, and beyond.
\newblock {\em J. Funct. Anal.}, 270(7):2456--2482, 2016.

\bibitem{Catbord}
P.~Cattiaux.
\newblock Regularit\'{e} au bord pour les densit\'{e}s et les densit\'{e}s
  conditionnelles d'une diffusion r\'{e}fl\'{e}chie hypoelliptique.
\newblock {\em Stochastics}, 20(4):309--340, 1987.

\bibitem{Catfourier}
P.~Cattiaux.
\newblock Stochastic calculus and degenerate boundary value problems.
\newblock {\em Ann. Inst. Fourier (Grenoble)}, 42(3):541--624, 1992.

\bibitem{CGsemin}
P.~Cattiaux and A.~Guillin.
\newblock Semi log-concave {M}arkov diffusions.
\newblock In {\em S\'{e}minaire de {P}robabilit\'{e}s {XLVI}}, volume 2123 of
  {\em Lecture Notes in Math.}, pages 231--292. Springer, Cham, 2014.

\bibitem{CGlogconc}
P.~Cattiaux and A.~Guillin.
\newblock On the {P}oincar\'e constant of log-concave measures.
\newblock In {\em Geometric aspects of functional analysis, Israel Seminar
  2017-2019, Vol.1}, volume 2256 of {\em Lecture Notes in Math.}, pages
  171--217. Springer, Berlin, 2020.

\bibitem{CGperturb}
P.~Cattiaux and A.~Guillin.
\newblock Functional inequalities for perturbed measures with applications to
  log-concave measures and to some {B}ayesian problems.
\newblock To appear in {\it Bernoulli}. Available on Math. ArXiv 2101.11257
  [math PR], 2021.

\bibitem{CGR}
P.~Cattiaux, A.~Guillin, and C.~Roberto.
\newblock Poincar\'e inequality and the {$L^p$} convergence of semi-groups.
\newblock {\em Elec. Comm. in Prob.}, 15:270--280, 2010.

\bibitem{CGWrad}
P.~Cattiaux, A.~Guillin, and L.~Wu.
\newblock Poincar\'e and {L}ogarithmic {S}obolev inequalities for nearly radial
  measures.
\newblock preliminary version available on Math. ArXiv 1912.10825 [math PR].
  The revised one to appear in Acta Math. Sin. on the homepage of the first
  named author., 2021.

\bibitem{CGZ}
P.~Cattiaux, A.~Guillin, and P.~A. Zitt.
\newblock Poincar\'e inequalities and hitting times.
\newblock {\em Ann. Inst. Henri Poincar\'e. Prob. Stat.}, 49(1):95--118, 2013.

\bibitem{CL1}
P.~Cattiaux and C.~L\'{e}onard.
\newblock Minimization of the {K}ullback information of diffusion processes.
\newblock {\em Ann. Inst. H. Poincar\'{e} Probab. Statist.}, 30(1):83--132,
  1994.

\bibitem{CL1c}
P.~Cattiaux and C.~L\'{e}onard.
\newblock Correction to: ``{M}inimization of the {K}ullback information of
  diffusion processes'' [{A}nn. {I}nst. {H}. {P}oincar\'{e} {P}robab.
  {S}tatist. {\bf 30} (1994), no. 1, 83--132; {MR}1262893 (95d:60056)].
\newblock {\em Ann. Inst. H. Poincar\'{e} Probab. Statist.}, 31(4):705--707,
  1995.

\bibitem{CL3}
P.~Cattiaux and C.~L\'{e}onard.
\newblock Minimization of the {K}ullback information for some {M}arkov
  processes.
\newblock In {\em S\'{e}minaire de {P}robabilit\'{e}s, {XXX}}, volume 1626 of
  {\em Lecture Notes in Math.}, pages 288--311. Springer, Berlin, 1996.

\bibitem{Fre}
J.~Frehse.
\newblock Essential selfadjointness of singular elliptic operators.
\newblock {\em Bol. Soc. Brasil. Mat.}, 8(2):87--107, 1977.

\bibitem{Fuku}
M.~Fukushima.
\newblock {\em Dirichlet forms and {M}arkov processes}, volume~23 of {\em
  North-Holland Mathematical Library}.
\newblock North-Holland Publishing Co., Amsterdam-New York; Kodansha, Ltd.,
  Tokyo, 1980.

\bibitem{IW}
N.~Ikeda and S.~Watanabe.
\newblock {\em Stochastic differential equations and diffusion processes},
  volume~24 of {\em North-Holland Mathematical Library}.
\newblock North-Holland Publishing Co., Amsterdam-New York; Kodansha, Ltd.,
  Tokyo, 1981.

\bibitem{Klartuncond}
B.~Klartag.
\newblock A {Berry-Esseen} type inequality for convex bodies with an
  unconditional basis.
\newblock {\em Probab. Theory Relat. Fields}, 145(1-2):1--33, 2009.

\bibitem{ledspin}
M.~Ledoux.
\newblock Logarithmic {S}obolev inequalities for unbounded spin systems
  revisited.
\newblock In {\em S\'{e}minaire de {P}robabilit\'{e}s, {XXXV}}, volume 1755 of
  {\em Lecture Notes in Math.}, pages 167--194. Springer, Berlin, 2001.

\bibitem{ledgap}
M.~Ledoux.
\newblock Spectral gap, logarithmic {S}obolev constant, and geometric bounds.
\newblock In {\em Surveys in differential geometry.}, volume~IX, pages
  219--240. Int. Press, Somerville {MA}, 2004.

\bibitem{ledlogconc}
M.~Ledoux.
\newblock From concentration to isoperimetry: semigroup proofs.
\newblock In {\em Concentration, functional inequalities and isoperimetry},
  volume 545 of {\em Contemp. Math.}, pages 155--166. Amer. Math. Soc.,
  Providence, RI, 2011.

\bibitem{ledgamma}
M.~Ledoux.
\newblock $\gamma_2$ and ${\Gamma}_2$. in honour of {D}. {B}akry and {M}.
  {T}alagrand.
\newblock https://perso.math.univ-toulouse.fr/ledoux/publications-3/, 2015.

\bibitem{Lig}
T.~M. Liggett.
\newblock {$L\sp 2$} rates of convergence for attractive reversible nearest
  particle systems.
\newblock {\em Ann. Probab.}, 19:935--959, 1991.

\bibitem{MZ2}
P.-A. Meyer and W.~A. Zheng.
\newblock Tightness criteria for laws of semimartingales.
\newblock {\em Ann. Inst. H. Poincar\'{e} Probab. Statist.}, 20(4):353--372,
  1984.

\bibitem{MZ1}
P.-A. Meyer and W.~A. Zheng.
\newblock Construction de processus de {N}elson r\'{e}versibles.
\newblock In {\em S\'{e}minaire de probabilit\'{e}s, {XIX}, 1983/84}, volume
  1123 of {\em Lecture Notes in Math.}, pages 12--26. Springer, Berlin, 1985.

\bibitem{emil1}
E.~Milman.
\newblock On the role of convexity in isoperimetry, spectral-gap and
  concentration.
\newblock {\em Invent. math.}, 177:1--43, 2009.

\bibitem{OR}
F.~Otto and M.~G. Reznikoff.
\newblock A new criterion for the logarithmic {S}obolev inequality and two
  applications.
\newblock {\em J. Funct. Anal.}, 243(1):121--157, 2007.

\bibitem{PW}
\'{E}. Pardoux and R.~J. Williams.
\newblock Symmetric reflected diffusions.
\newblock {\em Ann. Inst. H. Poincar\'{e} Probab. Statist.}, 30(1):13--62,
  1994.

\bibitem{rw}
M.~R{\"o}ckner and F.~Y. Wang.
\newblock Weak {P}oincar\'e inequalities and {$L\sp 2$}-convergence rates of
  {M}arkov semigroups.
\newblock {\em J. Funct. Anal.}, 185(2):564--603, 2001.

\bibitem{barrous}
O.~Roustant, F.~Barthe, and B.~Ioos.
\newblock Poincar\'e inequalities on intervals - application to sensitivity
  analysis.
\newblock {\em Electron. J. Stat.}, 11(2):3081--3119, 2017.

\bibitem{RoyF}
G.~Royer.
\newblock {\em Une initiation aux in\'{e}galit\'{e}s de {S}obolev
  logarithmiques}, volume~5 of {\em Cours Sp\'{e}cialis\'{e}s [Specialized
  Courses]}.
\newblock Soci\'{e}t\'{e} Math\'{e}matique de France, Paris, 1999.

\bibitem{RoyE}
G.~Royer.
\newblock {\em An initiation to logarithmic {S}obolev inequalities}, volume~14
  of {\em SMF/AMS Texts and Monographs}.
\newblock American Mathematical Society, Providence, RI; Soci\'{e}t\'{e}
  Math\'{e}matique de France, Paris, 2007.
\newblock Translated from the 1999 French original by Donald Babbitt.

\bibitem{W00}
F.Y. Wang.
\newblock Functional inequalities for empty essential spectrum.
\newblock {\em J. Funct. Anal.}, 170(1):219--245, 2000.

\bibitem{Wiel}
N.~Wielens.
\newblock The essential selfadjointness of generalized {S}chr\"{o}dinger
  operators.
\newblock {\em J. Funct. Anal.}, 61(1):98--115, 1985.

\end{thebibliography}

\end{document}